\documentclass[11pt]{article}
\usepackage{amsmath,amssymb,amsthm}
\usepackage{graphicx,enumerate}
\usepackage{pdfsync,epstopdf,comment,url}
\usepackage[x11names,usenames,dvipsnames,svgnames,table,rgb]{xcolor}
\usepackage{pgf,tikz,pgfplots}
\usetikzlibrary{arrows}
\usetikzlibrary{arrows.meta}

\usepackage{tikz-3dplot}
\usepackage{hyperref}

\usepackage{subcaption}
\numberwithin{equation}{section}
\usepackage[top=0.5in, bottom=0.8in, left=0.7in, right=0.7in, marginparwidth=1.5in]{geometry}
\newtheorem{theorem}{Theorem}[section]

\newtheorem{proposition}{Proposition}[section]
\newtheorem{lemma}{Lemma}[section]
\theoremstyle{definition}
\newtheorem{remark}{Remark}[section]
\newtheorem{definition}{Definition}[section]

\newtheorem{assumption}{Assumption}

\newcommand{\ie}{{i.e.}, }

\newcommand{\nn}{\mathbb{N}} %Nonnegative integers
\newcommand{\norm}[1]{\left\Vert {#1} \right\Vert} %Norm
\newcommand{\erl}{\left(-\infty , +\infty\right]} %Extended real line
\newcommand{\dom}[1]{\mathrm{dom}\,{#1}} %Domain
\newcommand{\idom}[1]{\mathrm{int\,dom}\,{#1}} %Interior of the Domain
\newcommand{\cdom}[1]{\overline{\mathrm{dom}\,{#1}}} %Closure of the Domain
\newcommand{\crit}[1]{\mathrm{crit}\,{#1}}
\newcommand{\dist}{\mathrm{dist}} %Distance between point and set
\newcommand{\act}[1]{\left\langle {#1} \right\rangle} %The value of 1
\newcommand{\seq}[2]{\{{#1}_{{#2}}\}_{{#2} \in \mathbb{N}}}
\newcommand{\Seq}[2]{\{{#1}^{{#2}}\}_{{#2} \in \mathbb{N}}}

\newcommand{\limit}[2]{\lim_{{#1} \rightarrow {#2}}}
\newcommand{\argmin}{\mathrm{argmin}}

\newcommand{\sgn}{\mathrm{sgn}}
\newcommand{\prox}{\mathrm{prox}} %Proximal map
\newcommand{\barc}{\overline{C}}

\newcommand{\PPP}{\mathcal{P}}
\newcommand{\SSS}{\mathcal{S}}
\newcommand{\R}{\mathbb{R}}

\newcommand{\bo}{{\bf 0}}

\newcommand{\real}{\mathbb{R}} %Real numbers
\graphicspath{{./figures/}}
\newcommand{\uL}{\bar{L}}
\newcommand{\lL}{\underline{L}}

\usepackage[colorinlistoftodos,prependcaption,textsize=tiny]{todonotes}

\title{ \textbf{Convex-Concave Backtracking for Inertial Bregman Proximal Gradient Algorithms in Non-Convex Optimization}}
\author{Mahesh Chandra Mukkamala\footnote{Faculty of Mathematics and Computer Science, Saarland University, 66123 Saarbr\"{u}cken, Germany, E-mail: \texttt{mukkamala@math.uni-sb.de}} \and Peter Ochs\footnote{Faculty of Mathematics and Computer Science, Saarland University, 66123 Saarbr\"{u}cken, Germany, E-mail: \texttt{ochs@math.uni-sb.de}} \and Thomas Pock\footnote{Institute of Computer Graphics and Vision, Graz University of Technology, 8010 Graz, Austria. E-mail: \texttt{pock@icg.tugraz.at}} \and Shoham Sabach\footnote{Faculty of Industrial Engineering, The Technion, Haifa, 3200003, Israel. E-mail: \texttt{ssabach@ie.technion.ac.il.}}}
\date{\vspace{-6ex}}

\begin{document}
\maketitle

\begin{abstract}
 	Backtracking line-search is an old yet powerful strategy for finding a better step sizes to be used in proximal gradient algorithms. The main principle is to locally find a simple convex upper bound of the objective function, which in turn controls the step size that is used. In case of inertial proximal gradient algorithms, the situation becomes much more difficult and usually leads to very restrictive rules on the extrapolation parameter. In this paper, we show that the extrapolation parameter can be controlled by locally finding also a simple concave lower bound of the objective function. This gives rise to a double convex-concave backtracking procedure which allows for an adaptive choice of both the step size and extrapolation parameters. We apply this procedure to the class of inertial Bregman proximal gradient methods, and prove that any sequence generated by these algorithms converges globally to a critical point of the function at hand. Numerical experiments on a number of challenging non-convex problems in image processing and machine learning were conducted and show the power of combining inertial step and double backtracking strategy in achieving improved performances.
	\end{abstract}

	\medskip


\section{Introduction} \label{Sec:Intro}
	In this work we are interested in tackling non-convex additive composite minimization problems, which include the sum of two extended-valued functions: a non-smooth function denoted by $f$ (possibly non-convex) and a smooth function denoted by $g$ (possibly non-convex). More precisely, we consider problems of the following form
	\begin{equation*}
		(\PPP) \qquad \inf \left\{ \Psi\left(x\right) \equiv f\left(x\right) + g\left(x\right) : \, x \in \barc \right\},
	\end{equation*}
	where $\barc$ is a nonempty, closed and convex set in $\real^{d}$. We will give a more precise statement in Section \ref{Sec:Bregman} about the involved functions and set. There is a tremendous number of applications in machine learning, computer vision, statistics, and many more, that can be formulated in this framework.
\medskip

	Motivated by challenging applications as illustrated in Section \ref{Sec:Numerics}, we consider here an instance of problem $(\PPP)$, where the smooth function $g$ has a gradient that is not necessarily globally Lipschitz continuous. The restrictive assumption of having Lipschitz continuous gradient can be replaced with a certain convexity condition, which was proposed and developed first in \cite{BBT2016} for problems $(\PPP)$ with convex functions, and recently extended to the non-convex setting in \cite{BSTV2018}. More details on these recent developments will be given below in Section \ref{Sec:Bregman}.
\medskip

	This convexity condition easily yields an approximation of the objective function at hand by a convex function from above (majorant) and a concave function from below (minorant). In the traditional setting, where the gradient of the smooth function $g$ is Lipschitz continuous, the majorant and the minorant are quadratic functions. In this case, it is well-known that the tightness of the quadratic approximations is directly related to restrictions on the step size to be used in the algorithm. The same relation is true for the convexity condition. In addition to their global existence, these approximations can be locally improved by backtracking (line search) strategies and it is well-known that tight approximations are advantageous, as we explain below in more detail.
\medskip

	Interestingly, while the step size is usually restricted by the quality of the majorant, the extrapolation (also known as inertia or over-relaxation) parameter is also affected by the quality of the minorant. This observation suggests to adapt the majorant and the minorant independently. In this paper we propose an efficient backtracking strategy that locally determines a tight majorant and minorant to exploit as much information as possible from the objective function, to be used in the proposed algorithm. This leads to a highly efficient algorithm, which is able to detect ``the degree of local convexity'' of the objective function (see Section \ref{Sec:CoCaIn} for details). As the backtracking procedure seeks for tight convex majorants and concave minorants, our idea is to combine it with an inertial step. We propose an inertial version of the Bregman Proximal Gradient (BPG) algorithm, which uses a convex-concave backtracking procedure to dynamically adjust the step size and the extrapolation parameter. Therefore, we call our algorithm \textit{Convex-Concave Inertial} BPG (CoCaIn BPG in short). We prove a global convergence result of this algorithm (see Section \ref{sec:summary-of-results} for an overview of the results and Section \ref{Sec:Convergence} for the details) to critical points of the objective function. The efficiency, which we demonstrate on several practical applications, comes from combining the inertial step with the novel \textit{convex-concave backtracking strategy}, which fully exploits the power of tight local approximations in achieving large step sizes and large extrapolation parameters that can be used at the same time.
\medskip

	Before concluding this section, we would like to give the reader a first intuition about the convex-concave backtracking strategy on a simple instance of problem $(\PPP)$. 

\paragraph{A simple illustrative example.} 
	In the following, we consider the following particular instance of problem $(\PPP)$: $C = \real^{d}$, $f \equiv 0$ and the gradient of $g$ is $L$-Lipschitz continuous. Even in this simpler setting, the convex-concave backtracking strategy is novel. 
\medskip

  	In this smooth and non-convex setting, an update step of a classical inertial based gradient method, starting with some $x^{0} \in \real^{d}$, reads as follows
	\begin{align*}
		y^{k} & = x^{k} + \gamma_{k}\left(x^{k} - x^{k - 1}\right), \\
   		x^{k + 1} & = y^{k} - \frac{1}{\uL_{k}} \nabla g\left(y^{k}\right),
	\end{align*}
	where $\gamma_{k} \in \left[0 , 1\right]$, $k \in \nn$, is an extrapolation parameter and $\uL_{k} > 0$. If $g$ is convex and the extrapolation parameter $\gamma_{k}$ is carefully chosen, this recovers the popular Nesterov Accelerated Gradient method \cite{N1983} (for $f \neq 0$, again in the convex setting, see \cite{BT2009}). It is well-known that the gradient step above, can be equivalently written as follows
	\begin{equation*}
    		x^{k + 1} = \argmin_{x \in \real^{d}} \left\{ g\left(y^{k}\right) + \act{\nabla g\left(y^{k}\right) , x - y^{k}} + \frac{\uL_{k}}{2}\norm{x - y^{k}}^{2} \right\}.
	\end{equation*}
  	For a proper $\uL_{k}$, the function to be minimized above is a convex quadratic majorant of the function $g$ (due to the classical Descent Lemma), which is a property that is also crucial for the convergence analysis of the algorithm. Classically, $\uL_{k} \geq L$, $k \in \nn$, is a sufficient condition to guarantee the existence of a quadratic majorant. However, locally, i.e., between the points $y^{k}$ and $x^{k + 1}$, the parameter $\uL_{k}$ may be significantly smaller than the global Lipschitz constant $L$ (which will immediately affect the step size of the algorithm). More precisely, note that the Descent Lemma,
	\begin{equation} \label{eq:intro-upper-bound}
   		\left| g\left(x\right) - g\left(y^{k}\right) - \act{\nabla g\left(y^{k}\right) , x - y^{k}} \right| \leq \frac{L}{2}\norm{x - y^{k}}^{2}, \qquad \forall \,\, x \in \real^{d},
	\end{equation}
	actually guarantees the existence of a quadratic minorant and a quadratic majorant that are determined by the same (global) parameter $L$. However, only the majorant limits the step size that is used in the algorithm. As shown in Figure \ref{fig:vis-upper-lower-bound}, tighter approximations can be computed if the parameters of the minorant and the majorant are allowed to differ:
  	\begin{equation} \label{eq:intro-upper-lower-bound}
    		-\frac{\lL_{k}}{2}\norm{x - y^{k}}^{2} \leq g\left(x\right) - g\left(y^{k}\right) - \act{\nabla g\left(y^{k}\right) , x - y^{k}} \leq \frac{\uL_{k}}{2}\norm{x - y^{k}}^{2},
  	\end{equation}
  	i.e., the minorant parameter $\lL_{k}$ could be different from the majorant parameter $\uL_{k}$. 

	\begin{figure}[htb]
		\begin{center}
			\newcommand{\thisTikzScaling}{0.75}
			\pgfplotsset{compat=1.14}
\definecolor{sexdts}{rgb}{0.1803921568627451,0.49019607843137253,0.19607843137254902}
\definecolor{dbwrru}{rgb}{0.8588235294117647,0.3803921568627451,0.0784313725490196}
\definecolor{wvvxds}{rgb}{0.396078431372549,0.3411764705882353,0.8235294117647058}
\definecolor{dtsfsf}{rgb}{0.8274509803921568,0.1843137254901961,0.1843137254901961}
\definecolor{rvwvcq}{rgb}{0.08235294117647059,0.396078431372549,0.7529411764705882}
\begin{tikzpicture}[line cap=round,line join=round,>=triangle 45,x=1cm,y=1cm, scale=\thisTikzScaling]
\begin{axis}[ticks=none,
x=1cm,y=1cm,
axis lines=middle,
xmin=-1.894141327016755,
xmax=13.057828409428886,
ymin=1.3,
ymax=9.724541431534774,
xtick={-8,-6,...,20},
ytick={-2,0,...,18},
axis line style={draw=none},
]
 \begin{scope}
\draw[-{Latex[black,length=5mm,width=2mm]},semithick] (-1.3,2) -- (13.057828409428886,2);
\draw[-{Latex[black,length=5mm,width=2mm]},semithick] (-0.5,-0.5) -- (-0.5,9.724541431534774);
\draw[line width=1.8pt,color=rvwvcq,smooth,samples=100,domain=-8.694141327016755:21.057828409428886] plot(\x,{abs((\x))+sin(((\x))*180/pi)});
\draw [samples=100,rotate around={-180:(9.26695047775061,8.921889627222484)},xshift=9.26695047775061cm,yshift=8.921889627222484cm,line width=1.2pt,dotted,color=dtsfsf,domain=-13.333333333333334:13.333333333333334)] plot (\x,{(\x)^2/2/1.6666666666666667});
\draw [samples=100,rotate around={0:(4.039829713349634,3.7988662236665114)},xshift=4.039829713349634cm,yshift=3.7988662236665114cm,line width=1.2pt,dotted,color=dbwrru,domain=-10:10)] plot (\x,{(\x)^2/2/1});
\draw [color=dbwrru](-0.3942713714029893,8.368126086115542) node[anchor=north west] {$g\left(y^k\right) + \act{\nabla g\left(y^k\right) , x- y^k} + \frac{{\bar L}_k}{2}\norm{x - y^k}^{2}$};
\draw [color=dtsfsf](5.60668468610532,5.181546115706323) node[anchor=north west] {$g\left(y^k\right) + \act{\nabla g\left(y^k\right) , x- y^k} -\frac{{\underline L}_k}{2}\norm{x - y^k}^{2}$};
\draw [line width=0.4pt,dotted] (7.5,2)-- (7.5,8.437999976774739);
\draw [line width=0.4pt,dotted] (6.8,2)-- (6.8,7.294113351138608);
\draw [line width=0.4pt,dotted] (6,2)-- (6,5.720584501801074);
\draw [line width=0.4pt,dotted] (4.04,2)-- (4.039829713349634,3.7988662236665105);
\begin{scriptsize}
\draw [fill=dtsfsf] (6,5.720584501801074) circle (2.5pt);
\draw [fill=rvwvcq] (6.8,7.294113351138608) circle (2.5pt);
\draw [fill=wvvxds] (7.5,8.437999976774739) circle (2.5pt);
\draw [fill=sexdts] (4.039829713349634,3.7988662236665105) ++(-2.5pt,0 pt) -- ++(2.5pt,2.5pt)--++(2.5pt,-2.5pt)--++(-2.5pt,-2.5pt)--++(-2.5pt,2.5pt);
\draw [fill=rvwvcq] (4.039829713349634,3.257599881654617) circle (2.5pt);
\draw[color=rvwvcq] (3.4102722159478214,2.759724971246093) node {$f(x^{k+1})$};
\draw [fill=dtsfsf] (6,2) circle (2.5pt);
\draw[color=dtsfsf] (5.794968901420688,1.6) node {$y^k$};
\draw [fill=rvwvcq] (6.8,2) circle (2.5pt);
\draw[color=rvwvcq] (6.713097143352288,1.6) node {$x^{k}$};
\draw [fill=wvvxds] (7.5,2) circle (2.5pt);
\draw[color=wvvxds] (7.533256004908299,1.6) node {$x^{k-1}$};
\draw [fill=rvwvcq] (4.04,2) circle (2.5pt);
\draw[color=rvwvcq] (3.6809018987379105,1.6) node {$x^{k+1}$};
\end{scriptsize}
\end{scope}
\end{axis}
\end{tikzpicture}
		\end{center}
		\caption{The inequalities in \eqref{eq:intro-upper-lower-bound} guarantee that the objective function has a quadratic concave minorant and a quadratic convex majorant. The proposed convex-concave backtracking strategy locally estimates both the lower and the upper approximations using a double backtracking procedure.}
		\label{fig:vis-upper-lower-bound}
	\end{figure}
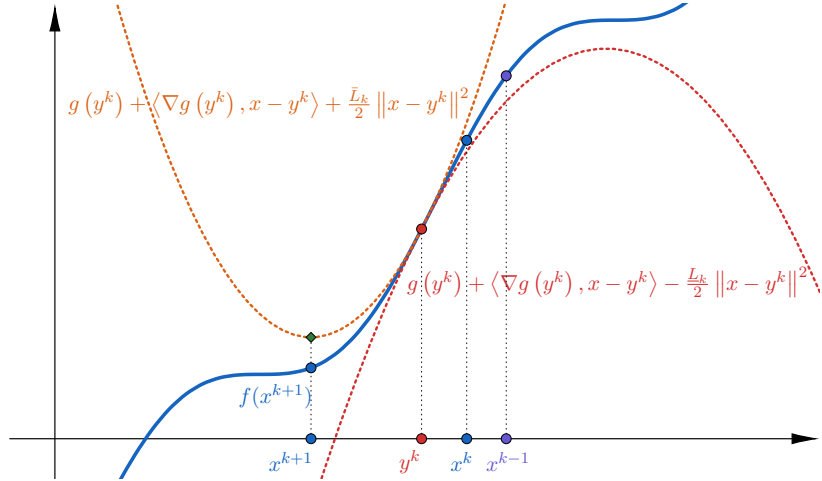

  	While the step size of the algorithm only depends on the majorant parameter $\uL_{k}$, the extrapolation parameter $\gamma_{k}$ also depends on the minorant parameter $\lL_{k}$. When $\uL_{k} = \uL$ and $\lL_{k} = \lL$, for all $k \in \nn$, it was established in \cite{WCP2017} that for any $0 \leq \gamma_{k} \leq \overline{\gamma}$, when 
	\begin{equation*}
		\overline{\gamma} < \sqrt{\frac{\uL}{\lL + \uL}}   \qquad \left(= \frac{1}{\sqrt{2}} \quad \text{for}\, \uL = \lL\right),
	\end{equation*}
	the generated sequence converges linearly (under certain error bound condition).
\medskip

	If the minorant parameter $\lL_{k}$ is close to $0$, which means that the function $g$ is ``locally convex'', the extrapolation parameter $\gamma_{k}$ can be taken close to $1$, which makes the algorithm we present ``similar" to an Accelerated Gradient method in the non-convex setting.
\medskip

	Below, we will show that using the minorant and the majorant in a local fashion (instead of their global counterparts) is very useful in developing the inertial Bregman Proximal Gradient method.
\bigskip
    
	{\bf Notation.} We use standard notation and concepts which, unless otherwise specified, can all be found in \cite{RW1998-B}.

\section{The Bregman Framework} \label{Sec:Bregman}
	In this section we will first recall the definition of Bregman distance, which stands at the heart of our developments. It was introduced in \cite{Bregman67} and popularized by \cite{Censor1981}.   Based on that we will shortly review the recent concept of smooth adaptable functions, which in some sense extends and generalizes the class of smooth functions with globally Lipschitz continuous gradient. Then, we will provide the basic and essential ingredients to deal with the Bregman Proximal Gradient method.
\medskip

	We begin with the notion of kernel generating distance functions, which was recently stated in \cite{BSTV2018} (in this respect see also \cite{AT2006}).
	\begin{definition}(Kernel Generating Distance) \label{D:KernelGen}
		Let $C$ be a nonempty, convex and open subset of $\real^{d}$. Associated with $C$, a function $h : \real^{d} \rightarrow \erl$ is called a \textit{kernel generating distance} if it satisfies the following:
    		\begin{itemize}
    			\item[$\rm{(i)}$] $h$ is proper, lower semicontinuous and convex, with $\dom h \subset \barc$ and $\dom \partial h = C$.
     			\item[$\rm{(ii)}$] $h$ is $C^{1}$ on $\idom h \equiv C$.
 			\end{itemize}
		We denote the class of kernel generating distances by $\mathcal{G}(C)$.
 	\end{definition}
	Given $h \in \mathcal{G}(C)$, the \textit{Bregman distance} that is associated to $h$, is a proximity measure $D_{h} : \dom h \times \idom h \to \real_{+}$ which is defined by
    \begin{equation*}
   		D_{h}\left(x , y\right) := h\left(x\right) - \left[h\left(y\right) + \act{\nabla h\left(y\right) , x - y}\right].
   	\end{equation*}
   	This object is not a distance according to the classical definition (for example, it is not symmetric in general). However, the Bregman distance between two points is nonnegative if and only if the function $h$ is convex. If $h$ is known to be strictly convex, we have that $D_{h}\left(x , y\right) = 0$ if and only if $x = y$. The classic example of a Bregman distance is the squared Euclidean distance, which is generated by $h(x) = \norm{x}^{2}$.	For more examples, results and applications of Bregman distances, see \cite{CZ1992,T1992,E1993,BB1997,T2018} and references therein. 
\medskip

   	An important property that is always crucial when dealing with Bregman distances is the well-known \textit{three-points identity} \cite[Lemma 3.1]{CT1993}: for any $y , z \in \idom h$ and $x \in \dom h$,
	\begin{equation} \label{ThreePI}
		D_{h}\left(x , z\right) - D_{h}\left(x , y\right) - D_{h}\left(y , z\right) = \act{\nabla h\left(y\right) - \nabla h\left(z\right) , x - y}.
	\end{equation}
	We conclude this part by restating \emph{our optimization model}
	\begin{equation*}
		(\PPP) \qquad \inf \left\{ \Psi \equiv f\left(x\right) + g\left(x\right) : \; x \in \barc \right\},
	\end{equation*}
	and making the first connection to the Bregman framework. One important feature of using Bregman distances in optimization algorithms is the ability of relate the constraint set $C$ to a certain kernel generating distances function $h \in \mathcal{G}(C)$. From now on, we make the following assumption.
	\begin{assumption} \label{A:AssumptionA}
		\begin{itemize}
    			\item[$\rm{(i)}$] $h \in \mathcal{G}(C)$ with $\barc = \cdom h$.
     		\item[$\rm{(ii)}$] $f : \real^{d} \rightarrow \erl$ is a proper and lower semicontinuous function (possibly non-convex) with $\dom f \cap C \neq \emptyset$.
    			\item[$\rm{(iii)}$] $g : \real^{d} \rightarrow \erl$ is a proper and lower semicontinuous function (possibly non-convex) with $\dom{h} \subset \dom{g}$, which is continuously differentiable on $C$.
        		\item[$\rm{(iv)}$] $v(\PPP) := \inf \left\{ \Psi\left(x\right) : \; x \in \barc \right\} > -\infty$.
    		\end{itemize}	
	\end{assumption}
	
\subsection{Smooth Adaptable Functions} \label{SSec:SmoothAF}
	One goal of this work is to deal with the non-convex optimization model ($\PPP$) where the gradient of the smooth function $g$ is not globally Lipschitz. Recently, Bauschke, Bolte and Teboulle \cite{BBT2016}, observed that the property of having a Lipschitz continuous gradient can be interpreted equivalently as a certain convexity condition on the function itself. This opens the gate for generalizing known results in the convex setting. It was extended to the non-convex setting in \cite{BSTV2018} with the concept of smooth adaptable functions given below.
	\begin{definition}[L-smooth Adaptable] \label{D:Smad}
		A pair $(g , h)$ is called $L$-smooth adaptable (\textbf{$L$-smad}) on $C$ if there exists $L > 0$ such that $Lh - g$ and $Lh + g$ are convex on $C$.
	\end{definition}
	The convexity requirement of $Lh + g$ can be written with respect to a different parameter $\ell \leq L$, which is key to the proposed double backtracking procedure to be developed in Section \ref{SSec:DoubleBP}. In this section, for the sake of simplicity, we use $\ell = L$.
\medskip

	The optimization model ($\PPP$) appears with a smooth term in the objective function which is very common in many fields of applications. A crucial pillar in designing and analyzing algorithms for tackling this model, is usually based on the fact that the smooth part in the objective function has a Lipschitz continuous gradient. This property, via the well-known Descent Lemma, guarantees us that a lower and an upper quadratic approximation exist. For $L$-smooth adaptable functions, we will use the following extended version of the Descent Lemma (see \cite[Lemma 2.1, p. 2134]{BSTV2018}).
	\begin{lemma}[Extended Descent Lemma] \label{L:ExtenDL}
		The pair of functions $(g , h)$ is $L$-smooth adaptable on $C$ if and only if:
		\begin{equation} \label{L:NoLipsDecent:1}
			\left| g\left(x\right) - g\left(y\right) - \act{\nabla g\left(y\right) , x - y} \right| \leq LD_{h}\left(x , y\right), \quad \forall \,\, x , y \in \idom h.
		\end{equation}
	\end{lemma}
	\begin{remark}[Invariance to Strong Convexity] \label{rem:sc-invariance}
		We would like to note that the $L$-smooth adaptable property is \textit{invariant} when $h$ is additionaly assumed to be  $\sigma$-strongly convex. Indeed, as described in \cite{BSTV2018}, since convexity of $g$ is not needed, we can define $\omega(x):= \left(\sigma_1/2\right)\norm{x}^{2}$, and then for any $0 < \sigma_{1} < \sigma$, we have
		\begin{equation*}
			Lh - g = L\left(h - \omega\right) - \left(g - L\omega\right) := L {\bar h} - {\bar g},
		\end{equation*}
		namely, the new pair $\left( {\bar g} , {\bar h}\right)$ satisfies the \textbf{L-smad} property on $C$.
	\end{remark}

\subsection{The Bregman Proximal Gradient Algorithm} \label{Sec:BregmanPG}
	In this section we review the basic notations and results needed to study Bregman based optimization methods. We first recall the definition of the Bregman proximal mapping \cite{T1992}, which is associated with a proper and lower semi-continuous function $f : \real^{d} \rightarrow \erl$, and is defined by
	\begin{equation*}
		\prox_{f}^{h}\left(x\right) \in \argmin \left\{ f\left(u\right) + D_{h}\left(u , x\right) : \, u \in \real^{d} \right\}, \quad \forall \,\, x \in \idom h.
	\end{equation*}
	With $h \equiv \left(1/2\right)\norm{\cdot}^{2}$, the above boils down to the classical set-valued \textit{Moreau proximal mapping} introduced in \cite{M1965}. We refer the reader to the recent survey paper \cite{T2018}, and references therein. Here, we will focus on the Bregman proximal gradient mapping, which will take a central role in the algorithm to be developed in the next section. Given $x \in \idom h$ and a step size parameter $\tau > 0$, the \textit{Bregman proximal gradient} mapping is defined by
	\begin{align}
		T_{\tau}\left(x\right) & \in  \argmin \left\{ f\left(u\right) + \act{\nabla g\left(x\right) , u - x} + \frac{1}{\tau} D_{h}\left(u , x\right) : \, u \in \barc \right\} \nonumber \\
		& = \argmin \left\{ f\left(u\right) + \act{\nabla g\left(x\right) , u - x} + \frac{1}{\tau} D_{h}\left(u , x\right) : \, u \in \real^{d} \right\}, \label{D:OperT}		 	
	\end{align}
	where the second equality follows from the fact that $\dom h \subset \barc$. Note that here with $h \equiv \left(1/2\right)\norm{\cdot}^{2}$, the above recovers the classical proximal gradient mapping. Since $f$ could be non-convex, the mapping $T_{\tau}$ is not, in general, single-valued. This mapping emerges from the usual approach, which consists of linearizing the differentiable function $g$ around a point $x$ and regularizing it with a proximal distance from that point. Similar to \cite{BSTV2018}, the following assumption guarantees that the Bregman proximal gradient mapping is well-defined.
	\begin{assumption} \label{A:AssumptionB}
		\begin{itemize}
			\item[$\rm{(i)}$] The function $h + \tau f$ is supercoercive for all $\tau > 0$, that is,
				\begin{equation*}
					\lim_{\norm{u} \rightarrow \infty} \frac{h\left(u\right) + \tau f\left(u\right)}{\norm{u}} = \infty.
				\end{equation*}		
			\item[$\rm{(ii)}$] For all $x \in C$,  we have $T_{\tau}\left(x\right) \subset C$.
		\end{itemize}			
	\end{assumption}
	Assumption \ref{A:AssumptionB}(i) is a standard coercivity condition, which is for instance automatically satisfied when $\barc$ is compact. On the other hand, Assumption \ref{A:AssumptionB}(ii) can be shown to hold under a classical constraint qualification condition. It also holds automatically when $f$ is convex or when $C = \real^{d}$. The following result from \cite{BSTV2018}, ensures that the Bregman proximal gradient mapping is well-defined.
	\begin{lemma}[Well-Posedness of $T_{\tau}$] \label{P:WellProximal}
		Suppose that Assumptions \ref{A:AssumptionA} and \ref{A:AssumptionB} hold, and let $x \in \idom h$. Then, the set $T_{\tau}\left(x\right)$ is a nonempty and compact subset of $\idom h$.
	\end{lemma}
	
\section{The Inertial Bregman Proximal Gradient Method} \label{Sec:CoCaIn}
	Our proposed algorithm belongs to the class of inertial based optimization methods. The most well-known method in this class is the so-called Heavy-ball method, which was introduced by Polyak \cite{P1964} to minimize convex and smooth functions. A popular variant of the method, when applied to the additive composite model ($\PPP$) with $C = \real^{d}$, takes the following form. Start with any $x^{0} = x^{1} \in \real^{d}$, and generate iteratively a sequence $\Seq{x}{k}$ via
	\begin{align}
		y^{k} & = x^{k} + \gamma_{k}\left(x^{k} - x^{k - 1}\right), \label{InertialPG:1} \\
		x^{k + 1} & \in \argmin_{u} \left\{ f\left(u\right) + \act{\nabla g\left(y^{k}\right) , u - y^{k}} + \frac{1}{2\tau_{k}}\norm{u - y^{k}}^{2} \right\}, \label{InertialPG:2}
	\end{align}
  	where $\gamma_{k} \in \left[0 , 1\right]$ is an \textit{extrapolation} parameter and $\tau_{k} > 0$ is a \textit{step size} paramter. In \cite{OCBP2014}, an inertial proximal gradient algorithm, called \textit{iPiano}, was proposed\footnote{With a small modification that the proximity term is centered around the extrapolated point $y^{k}$, while the gradient of $g$ is evaluated at $x^{k}$.}. It was shown that under Assumption \ref{A:AssumptionA}, if $f$ is convex and $g$ has a globally Lipschitz continuous gradient, the sequence $\Seq{x}{k}$ converges globally to a critical point (in this setting, under additional error-bound condition, a linear rate of convergence was proved in \cite{WCP2017}). The case where also the function $f$ is not necessarily convex was treated in \cite{BRE16,Ochs15}. Two years later, in \cite{PS2016} a block version of the method, called \textit{iPALM} was proposed and analyzed in the fully non-convex setting, \ie both $f$ and $g$ are non-convex. In this case, a global convergence result to critical points was also established. A unified analysis was presented in \cite{Ochs2019}. \medskip

	In this work we propose a Bregman variant of the method mentioned above (see steps \eqref{InertialPG:1} and \eqref{InertialPG:2}), which also handle the two involved parameter $\gamma_{k}$ and $\tau_{k}$, $k \in \nn$, in a dynamic fashion. To this end we incorporate into our basic steps two routines aiming at controlling and updating these parameters.
  	
\subsection{The Convex-Concave Backtracking Procedure} \label{SSec:DoubleBP}
	As we already illustrated on a simple example in the introduction, the origin of this procedure comes from the fact that for smooth adaptable functions we can build lower and upper approximations as given in Lemma \ref{L:ExtenDL}:
	\begin{equation} \label{Approximations}
		-\lL D_{h}\left(x , y\right) \leq g\left(x\right) - g\left(y\right) - \act{\nabla g\left(y\right) , x - y} \leq \uL D_{h}\left(x , y\right), \quad \forall \,\, x , y \in \idom h.
	\end{equation}  	
  	Even though the existence of the parameters $\lL$ and $\uL$ could be globally guaranteed, in practice it is often difficult or computationally expensive to evaluate them. In such cases it is  recommended to apply a backtracking procedure that can locally verify the validity of the inequalities given in \eqref{Approximations}. However, in most cases only the upper approximation and the corresponding parameter $\uL$ are used. Here, we will develop a double backtracking procedure that locally verifies both the lower and the upper approximations, in order to better control and update the extrapolation parameter $\gamma_{k}$ and the step size parameter $\tau_{k}$ at each iteration $k \in \nn$. To the best of our knowledge, this is the first attempt to use the lower approximation in algorithms for tackling non-convex problems. It should be noted that in the case that $g$ is convex we have by definition $\lL = 0$, or even a convex quadratic lower approximation can be found when $g$ is strongly convex (see \cite{T2018} for a discussion and references about a strong convexity property with respect to a Bregman distance). 
	Based on the concepts described above, we will make the following additional assumptions on the involved functions.
	\begin{assumption} \label{A:AssumptionC}
		\begin{itemize}
			\item[$\rm{(i)}$] The function $h : \real^{d} \rightarrow \erl$ is $\sigma$-strongly convex on $C$.
			\item[$\rm{(ii)}$] The pair of functions $(g , h)$ is $L$-smooth adaptable on $C$.
			\item[$\rm{(iii)}$] 	There exists $\alpha \in \real$ such that $f\left(\cdot\right) - \left(\alpha/2\right)\norm{\cdot}^{2}$ is convex\footnote{Such functions are called semi-convex with modulus $\alpha$ (see \cite{Ochs15, Ochs18}).}.
		\end{itemize}			
	\end{assumption}
	A few comments on the assumption above are now in order. The first item is related to Remark~\ref{rem:sc-invariance}, which says that the smooth adaptable property is invariant to strongly convex kernel generating distance functions $h$. The third assumption allows us to deal with non-convex functions $f$ since $\alpha$ could be negative. See  Section \ref{Sec:Numerics} for examples of functions that satisfy all these assumptions. Now we are ready to present our algorithm, which is called Convex-Concave Inertial (CoCaIn) Bregman Proximal Gradient.

	\begin{figure}[h]
		\centering % changing from Ivory2 to Orange!10
		\fcolorbox{black}{Orange!10}{\parbox{15cm}{{\bf Convex-Concave Inertial BPG} \\
			{\bf Input.} $\delta, \varepsilon > 0$ with $1>\delta > \varepsilon$. \\
			{\bf Initialization.} $x^{0} = x^{1} \in \idom h \cap \dom f$, ${\bar L}_{0} > \frac{-\alpha}{(1-\delta)\sigma}$ and $\tau_{0} \leq {\bar L}_{0}^{-1}$. \\
			{\bf General Step.} For $k = 1 , 2 , \ldots$, compute
				\begin{align}
					y^{k} & = x^{k} + \gamma_{k}\left(x^{k} - x^{k - 1}\right) \in \idom h, \label{CoCaIn:1}
				\end{align}
				where $\gamma_{k}$ is chosen such that
				\begin{equation} \label{CoCaIn:3}
					\left(\delta - \varepsilon\right)D_{h}\left(x^{k - 1} , x^{k}\right) \geq \left(1 + {\underline L_{k}}\tau_{k-1}\right)D_{h}\left(x^{k} , y^{k}\right)
				\end{equation}
				holds and such that $\lL_{k}$ satisfies
				\begin{equation} \label{CoCaIn:4}
					g\left(x^{k}\right) \geq g\left(y^{k}\right) + \act{\nabla g\left(y^{k}\right) , x^{k} - y^{k}} - \lL_{k}D_{h}\left(x^{k} , y^{k}\right).
				\end{equation}
				Now, choose $\uL_{k} \geq \uL_{k - 1}$, set $\tau_{k} \leq \min \left\{ \tau_{k - 1} , \uL_{k}^{-1} \right\}$ and compute
				\begin{equation} \label{CoCaIn:2}
					x^{k + 1} \in  \argmin_{u} \left\{ f\left(u\right) + \act{\nabla g\left(y^{k}\right) , u - y^{k}} + \frac{1}{\tau_{k}}D_{h}\left(u , y^{k}\right) \right\} 										\end{equation}
				with $\uL_{k}$ fulfilling
				\begin{equation} \label{CoCaIn:5}
					g\left(x^{k + 1}\right) \leq g\left(y^{k}\right) + \act{\nabla g\left(y^{k}\right) , x^{k + 1} - y^{k}} + \uL_{k}D_{h}\left(x^{k + 1} , y^{k}\right) .
				\end{equation}
				}}
				\label{alg:cocain}
	\end{figure}
\medskip

	The two input parameters $\delta$ and $\varepsilon$ are free to be chosen by the user. As we will see later the parameter $\varepsilon$ measures the descent to be achieved at each iteration of the algorithm. 
\medskip

	 The steps \eqref{CoCaIn:1} and \eqref{CoCaIn:2} are the classical steps of the inertial proximal gradient method, while here since we are dealing with the Bregman variant, it must be guaranteed that the auxiliary vector $y^{k}$ as defined in \eqref{CoCaIn:1} belongs to $\idom h$. Otherwise the Bregman proximal gradient step \eqref{CoCaIn:2} is not defined (see Section \ref{Sec:BregmanPG}). Even though, in general, it is not easy to guarantee that, in our case this will not be an issue. Indeed, in order to derive global convergence results of Bregman based algorithms in the non-convex setting an essential assumption seems to be that the kernel generating distance function $h$ has a full domain, \ie $\dom h = \real^{d}$ (see, for instance, \cite{BSTV2018} for more details about this limitation). The steps \eqref{CoCaIn:4} and \eqref{CoCaIn:5} implement the double backtracking procedure (see Section \ref{SSec:Implement}). The step \eqref{CoCaIn:3} is designed to control the extrapolation parameter $\gamma_{k}$, $k \in \nn$, and should be validated at each iteration. However, a natural question would be if such a parameter always exists? We postpone the positive answer to this question, to Section \ref{SSec:Well}, and conclude this section with a list of our theoretical contributions.

\subsection{Summary of the Convergence Results} \label{sec:summary-of-results}
	Before we proceed with the well-posedness of CoCaIn BPG and the convergence analysis, we provide here a brief summary of  our results.
	\begin{itemize}
		\item We show the \textit{well-posedness of CoCaIn BPG}, in the sense that, one can always find  $\gamma_{k}$ such that \eqref{CoCaIn:3} is satisfied for all $k \in \nn$ (see Lemma \ref{L:Extra}). Moreover, we show that it suffices to know the Bregman symmetric coefficient $\alpha\left(h\right)$ (Definition \ref{D:SymmetricC}), in order to estimate the extrapolation parameter $\gamma_{k}$, $k \in \nn$.
		\item In the Euclidean setting, \ie when $h = \left(1/2\right)\norm{\cdot}^{2}$, we provide an \textit{explicit formula for the maximal extrapolation parameter}
			\begin{equation*}	
  				0 \leq \gamma_{k} \leq \overline{\gamma}, \qquad \overline{\gamma} < \sqrt{\frac{\uL_{k - 1}}{\uL_{k - 1} + \lL_{k}}},
			\end{equation*}
			which uses the majorant parameter $\uL_{k - 1}$ from the previous iterate, which is a key for the efficient implementation of the proposed convex-concave backtracking procedure. When $\uL_{k - 1} = \lL_{k}$, we easily recover that $\overline{\gamma} < 1/\sqrt{2}$.
		\item \textit{Stability and convergence of the objective function values of CoCaIn BPG}, which relies on finding an appropriate sequence of Lyapunov functions that enjoys a sufficient descent property (see Proposition \ref{prop:Lyapunov-1}).
		\item \textit{Global convergence of a sequence generated by the CoCaIn BPG method} to critical points of the objective function $\Psi$ (see Theorem \ref{T:GlobalCoCaInBPG}). This result relies on the concept of Gradient-like Descent Sequences (see Definition \ref{D:GradientLDS} below).
	\end{itemize}

\section{Well-Posedness of CoCaIn BPG} \label{SSec:Well}
	Now, we would like to verify the well-posedness of the CoCaIn BPG  algorithm. An important tool in achieving our goal is the recently introduced symmetry coefficient of a Bregman distance, which measures the lack of symmetry in $D_{h}\left(\cdot , \cdot\right)$, see \cite{BBT2016}.
	\begin{definition}[Symmetry Coefficient] \label{D:SymmetricC}
		Given $h \in \mathcal{G}(C)$, its \textit{symmetry coefficient} is defined by
		\begin{equation*}
			\alpha\left(h\right) := \inf \left\{ \frac{D_{h}\left(x , y\right)}{D_{h}\left(y , x\right)} : \, x , y \in \idom h, \, x \neq y \right\} \in \left[0 , 1\right].
		\end{equation*}
	\end{definition}
	An important	 and immediate consequence of this definition is the fact that for all $x , y \in \idom h$ we have
	\begin{equation} \label{E:SymmetricC}
		\alpha\left(h\right)D_{h}\left(x , y\right) \leq D_{h}\left(y , x\right) \leq \alpha\left(h\right)^{-1}D_{h}\left(x , y\right),
	\end{equation}
	where we have adopted the convention that $0^{-1} = +\infty$ and $+\infty \times r = +\infty$ for all $r \geq 0$. Clearly, the closer is $\alpha\left(h\right)$ to $1$, the more symmetric $D_{h}$ is  with perfect symmetry when $\alpha\left(h\right) = 1$ (which holds if and only if $h = \norm{\cdot}^{2}$).
\medskip

	To this end, we need to convince the reader about the existence of $\gamma_{k}$, $k \in \nn$, which satisfies \eqref{CoCaIn:3}, i.e., that
	\begin{equation*}
		\left(\delta - \varepsilon\right)D_{h}\left(x^{k - 1} , x^{k}\right) \geq \left(1 + {\underline L_{k}}\tau_{k-1}\right)D_{h}\left(x^{k} , y^{k}\right),
	\end{equation*}
	holds true. The following result provides a positive answer to the existences question and information on the relevant extrapolation parameters that satisfy this inequality.
	\begin{lemma}[General Extrapolation Behavior] \label{L:Extra}
		Given $h \in \mathcal{G}(C)$ with $\alpha\left(h\right) > 0$. Let $x_{1} , x_{2} , y \in \idom h$ and $y := x_{1} + \gamma\left(x_{1} - x_{2}\right)$ with $\gamma \geq 0$. Then, for a given $\kappa > 0$, there exists $\gamma^{\ast} >0$ such that
		\begin{equation} \label{eq:kappa-eq}
			D_{h}\left(x_{1} , y\right) \leq \kappa D_{h}\left(x_{2} , x_{1}\right), \quad \forall \,\, \gamma \in \left[0 , \gamma^{\ast}\right].
		\end{equation}
	\end{lemma}
	\begin{proof}
		From the three points identity (see \eqref{ThreePI}) we have 
		\begin{align*}
			D_{h}\left(y , x_{2}\right) & = D_{h}\left(y , x_{1}\right) + D_{h}\left(x_{1} , x_{2}\right) + \act{\nabla h\left(x_{1}\right) - \nabla h\left(x_{2}\right) , y - x_{1}} \\ 
			& = D_{h}\left(y , x_{1}\right) + D_{h}\left(x_{1} , x_{2}\right) + \gamma\act{\nabla h\left(x_{1}\right) - \nabla h\left(x_{2}\right) , x_{1} - x_{2}} \\ 
			& = D_{h}\left(y , x_{1}\right) + D_{h}\left(x_{1} , x_{2}\right) + \gamma\left(D_{h}\left(x_{1} , x_{2}\right) + D_{h}\left(x_{2} , x_{1}\right)\right).
		\end{align*}
		Now, from \eqref{E:SymmetricC}, we obtain that
		\begin{equation*}
			D_{h}\left(y , x_{2}\right) \leq \frac{1}{\alpha\left(h\right)}\left[D_{h}\left(x_{1} , y\right) + \left(\gamma\alpha\left(h\right) + 1 + \gamma\right)D_{h}\left(x_{2} , x_{1}\right)\right].
		\end{equation*}
		On the other hand, since $x_{1} = \left(y + \gamma x_{2}\right)/\left(1 + \gamma\right)$, we can use the fact that $u \rightarrow D_{h}\left(u , v\right)$, for a fixed $v \in \idom h$, is a convex function and therefore
		\begin{equation*}
			D_{h}\left(x_{1} , y\right) \leq \frac{\gamma}{1 + \gamma}D_{h}\left(x_{2} , y\right) \leq \frac{\gamma}{\alpha\left(h\right)\left(1 + \gamma\right)}D_{h}\left(y , x_{2}\right),
		\end{equation*}
		where the last inequality follows from \eqref{E:SymmetricC}. By combining the last two inequalities we derive that
		\begin{equation*}
			D_{h}\left(x_{1} , y\right) \leq \frac{\gamma}{\alpha\left(h\right)^{2}\left(1 + \gamma\right)}\left[D_{h}\left(x_{1} , y\right) + \left(\gamma\alpha\left(h\right) + 1 + \gamma\right)D_{h}\left(x_{2} , x_{1}\right)\right],
		\end{equation*}
		and, by re-arranging we have
		\begin{equation*}
			D_{h}\left(x_{1} , y\right) \leq \frac{\gamma\left(\gamma\alpha\left(h\right) + 1 + \gamma\right)}{\alpha\left(h\right)^{2}\left(1 + \gamma\right) - \gamma}D_{h}\left(x_{2} , x_{1}\right).
		\end{equation*}
		First, it is easy to verify that for $\gamma < \alpha\left(h\right)^{2}/\left(1 - \alpha\left(h\right)^{2}\right)$, the denominator is positive. In addition, to find $\gamma$ such that
		\begin{equation*}
			\frac{\gamma\left(\gamma\alpha\left(h\right) + 1 + \gamma\right)}{\alpha\left(h\right)^{2}\left(1 + \gamma\right) - \gamma} \leq \kappa,
		\end{equation*}
		we will use simple algebraic manipulations. Indeed, by re-arranging we have
		\begin{equation*}
			\gamma^{2}\underbrace{\left(\alpha\left(h\right) + 1\right)}_{a} + \gamma\underbrace{\left(1 + \kappa - \alpha\left(h\right)^{2}\kappa\right)}_{b} - \alpha\left(h\right)^{2}\kappa \leq 0.
		\end{equation*}
		Since $\alpha\left(h\right)^{2} \leq 1$, it follows that $b > 0$. We also have that $\Delta = b^{2} + 4a\alpha\left(h\right)^{2}\kappa > 0$, and thus there exists a positive root denoted by $\gamma^{\ast}$. Therefore, for any $\gamma \in \left[0 , \gamma^{\ast}\right]$, the desired result follows.
	\end{proof}
	\begin{remark} \label{rem:inertia-rem}
		Note that in the above lemma, $\gamma^{\ast}$ depends only on the symmetry coefficient $\alpha\left(h\right)$. Therefore, for the Euclidean distance with $\alpha\left(h\right) = 1$, this implies that,
		\begin{equation*}
			\gamma^{\ast} = \frac{-1 + \sqrt{1 + 8\kappa}}{4}\,.
		\end{equation*}
		However, for the Euclidean distance, the expression in \eqref{eq:kappa-eq}, can be simplified significantly. Indeed, since we take $h = (1/2)\norm{\cdot}^{2}$, then using the fact that $y^{k} - x^{k} = \gamma_{k}\left(x^{k} - x^{k - 1}\right)$ we obtain that $\gamma_{k} \leq \sqrt{\kappa}$. In the case of CoCaIn BPG, we have the following restriction on the maximal extrapolation parameter that can be used
		\begin{equation*}
			\gamma_{k} \leq \sqrt{\frac{\delta - \varepsilon}{1 + \lL_{k}\tau_{k - 1}}} \leq \sqrt{\frac{\left(\delta - \varepsilon\right)\uL_{k - 1}}{\uL_{k - 1} + \lL_{k}}}\,.
		\end{equation*}
		A related bound also appeared in \cite{WCP2017} as we discussed in the introduction. When, the values of $\lL_{k}$ and $\uL_{k - 1}$ are almost equal and $\delta - \varepsilon \approx 1$, then it is possible to choose the inertial parameter $\gamma_{k}$ such that $\gamma_{k} \approx 1/\sqrt{2}$. We discuss more about bounds of $\gamma_{k}$, $k \in \nn$, in Section \ref{SSec:without-backtracking}.
	\end{remark}

\section{Convergence Analysis of CoCaIn BPG} \label{Sec:Convergence}
	Before we proceed to the convergence analysis, we need the following technical lemma.
	\begin{lemma}[Function Descent Property] \label{L:Tech}
		Let $\Seq{x}{k}$ be a sequence generated by CoCaIn BPG. Then, for all $k \in \nn$, we have
		\begin{equation} \label{L:Tech:1}
			\Psi\left(x^{k}\right) \geq \Psi\left(x^{k + 1}\right) + \frac{1}{\tau_{k}}D_{h}\left(x^{k} , x^{k + 1}\right) + \frac{\alpha}{2}\norm{x^{k + 1} - x^{k}}^{2} - \left(\frac{1}{\tau_{k}} + {\underline L_{k}}\right)D_{h}\left(x^{k} , y^k\right).
		\end{equation}
	\end{lemma}
	\begin{proof}
		Fix $k \geq 1$. From the convexity of $f\left(\cdot\right) - \left(\alpha/2\right)\norm{\cdot}^{2}$, which holds thanks to Assumption \ref{A:AssumptionC}(iii), we obtain from the sub-gradient inequality \cite[Example 8.8 and Proposition 8.12]{RW1998-B} that
		\begin{equation*}
			f\left(x^{k}\right) - \frac{\alpha}{2}\norm{x^{k}}^{2} \geq f\left(x^{k + 1}\right) - \frac{\alpha}{2}\norm{x^{k + 1}}^{2} + \act{\xi^{k + 1} - \alpha x^{k + 1} , x^{k} - x^{k + 1}},
		\end{equation*}
		where $\xi^{k + 1} \in \partial f\left(x^{k + 1}\right)$. By rearranging the inequality we obtain
		\begin{equation} \label{L:Tech:1-1}
			f\left(x^{k}\right) \geq f\left(x^{k + 1}\right) + \frac{\alpha}{2}\norm{x^{k + 1} - x^{k}}^{2} + \act{\xi^{k + 1} , x^{k} - x^{k + 1}}.
		\end{equation}
		From the optimality condition of step \eqref{CoCaIn:2}, we have that
		\begin{equation*}
			\xi^{k + 1} + \nabla g\left(y^{k}\right) + \frac{1}{\tau_{k}}\left(\nabla h\left(x^{k + 1}\right) - \nabla h\left(y^{k}\right)\right) = \bo\,,
		\end{equation*}
		which combined with \eqref{L:Tech:1} yields that
		\begin{align*}
			f\left(x^{k}\right) & \geq f\left(x^{k + 1}\right) + \frac{\alpha}{2}\norm{x^{k + 1} - x^{k}}^{2} - \act{\nabla g\left(y^{k}\right) , x^{k} - x^{k + 1}} \\
			& + \frac{1}{\tau_{k}}\act{\nabla h\left(y^{k}\right) - \nabla h\left(x^{k + 1}\right) , x^{k} - x^{k + 1}} \\
			& = f\left(x^{k + 1}\right) + \frac{\alpha}{2}\norm{x^{k + 1} - x^{k}}^{2} - \act{\nabla g\left(y^{k}\right) , x^{k} - x^{k + 1}} \\
			& + \frac{1}{\tau_{k}}\left(D_{h}\left(x^{k} , x^{k + 1}\right) + D_{h}\left(x^{k + 1} , y^{k}\right) - D_{h}\left(x^{k} , y^{k}\right)\right),
		\end{align*}
		where the last equality follows from the three-points identity (see \eqref{ThreePI}). On the other hand, using the lower approximation given in \eqref{CoCaIn:4} and the upper approximation given in \eqref{CoCaIn:5}, we have that
		\begin{equation*}
			g\left(x^{k}\right) \geq g\left(x^{k + 1}\right) + \act{\nabla g\left(y^{k}\right) , x^{k} - x^{k + 1}} - {\underline L_{k}}D_{h}\left(x^{k} , y^{k}\right) - {\bar L_{k}}D_{h}\left(x^{k + 1} , y^{k}\right).
		\end{equation*}
		Combining the last two inequalities and using the fact that $\tau_{k}^{-1} \geq {\bar L_{k}}$, implies that
		\begin{equation*}
			\Psi\left(x^{k}\right) \geq \Psi\left(x^{k + 1}\right) + \frac{\alpha}{2}\norm{x^{k + 1} - x^{k}}^{2} + \frac{1}{\tau_{k}}D_{h}\left(x^{k} , x^{k + 1}\right) - \left(\frac{1}{\tau_{k}} + {\underline L_{k}}\right)D_{h}\left(x^{k} , y^{k}\right),
		\end{equation*}
		which completes the proof.
	\end{proof}
	Since we are dealing with inertial based methods, which belong to the class of non-descent methods, we can not expect to use classical convergence techniques for non-convex problems (see below for more information about it). In order to overcome the lack of descent, we will use the Lyapunov technique, which involves the construction of a sequence of new functions, which will be used to ``better" measure the progress of the algorithm, where by progress we mean a decrement in the Lyapunov function values. In several cases a trivial Lyapunov function would be to use the function itself, however in the case of non-descent methods, it is not a good choice, since it does not capture well the behavior of the iterates. The behavior of two subsequent iterates must be taken into consideration along with the function, as observed in \cite{OCBP2014, SBC2014}. 

\subsection{Lyapunov Function Descent Property of CoCaIn BPG}\label{SSec:lyapuno-bpg}
	Let $\Seq{x}{k}$ be a sequence generated by CoCaIn BPG. We define, at iterate $k \in \nn$, the following Lyapunov function
	\begin{equation} \label{eq:lyapunov-func}
		\Phi_{\delta}^{k}\left(x^{k} , x^{k - 1}\right) = \tau_{k - 1}\left(\Psi\left(x^{k}\right) - v(\PPP)\right) + \delta D_{h}\left(x^{k - 1} , x^{k}\right).
	\end{equation}
	This Lyapunov function involves two terms: (i) the term $\tau_{k - 1}\left(\Psi\left(x^{k}\right) - v(\PPP)\right)$, which measures the progress in original function values $\Psi$ with respect to the global optimal value of problem ($\PPP$) and (ii) the term given by $\delta D_{h}\left(x^{k - 1} , x^{k}\right)$, which ensures that the iterates stay close enough, with respect to the Bregman distance. Before we motivate further the usage of this Lyapunov function, we show its descent property.
	\begin{proposition} \label{prop:Lyapunov-1}
		Let $\Seq{x}{k}$ be a sequence generated by CoCaIn BPG. Then, for all $k \in \nn$, we have
		 \begin{equation} \label{prop:Lyapunov-1:1}
		 	\Phi_{\delta}^{k}\left(x^{k} , x^{k - 1}\right) \geq \Phi_{\delta}^{k + 1}\left(x^{k + 1} , x^{k}\right) + \varepsilon D_{h}\left(x^{k - 1} , x^{k}\right).
		\end{equation}
	\end{proposition}
	\begin{proof}
		Multiplying \eqref{L:Tech:1} with $\tau_{k}$, we obtain
		\begin{align*}
			\tau_{k}\left(\Psi\left(x^{k}\right) - v(\PPP)\right) & \geq \tau_{k}\left(\Psi\left(x^{k + 1}\right) - v(\PPP)\right) + \frac{\alpha\tau_{k}}{2}\norm{x^{k + 1} - x^{k}}^{2} + D_{h}\left(x^{k} , x^{k + 1}\right) \\
			& - \left(1 + {\underline L_{k}\tau_{k}}\right)D_{h}\left(x^{k} , y^{k}\right).
		\end{align*}
		By the definition of the Lyapunov function $\Phi_{\delta}^{k}$ and the fact that $\tau_{k} \leq \tau_{k - 1}$ we have
		\begin{align*}
			\Phi_{\delta}^{k}\left(x^{k} , x^{k - 1}\right) & \geq \Phi_{\delta}^{k + 1}\left(x^{k + 1} , x^{k}\right) + \frac{\alpha\tau_{k}}{2}\norm{x^{k + 1} - x^{k}}^{2} + \left(1 - \delta\right)D_{h}\left(x^{k} , x^{k + 1}\right) \\
			& + \delta D_{h}\left(x^{k - 1} , x^{k}\right) - \left(1 + {\underline L_{k}\tau_{k}}\right)D_{h}\left(x^{k} , y^{k}\right).
		\end{align*}	
		With $1 - \delta >0$ and the strong convexity of $h\left(\cdot\right)$, that follows from Assumption \ref{A:AssumptionC}(i), we obtain
		\begin{align*}
			\frac{\alpha\tau_{k}}{2}\norm{x^{k + 1} - x^{k}}^{2} + \left(1 - \delta\right)D_{h}\left(x^{k} , x^{k + 1}\right) \geq \left(\frac{\alpha\tau_{k}}{2} + \left(1 - \delta\right)\frac{\sigma}{2}\right)\norm{x^{k + 1} - x^{k}}^{2} \geq 0,
		\end{align*}
		where the last inequality holds, since $\tau_{k}^{-1} \geq {\bar L}_{k}$ and ${\bar L}_{k} \geq -\alpha/\left(1 - \delta\right)\sigma$. Next, we observe that
		\begin{equation*}
			D_{h}\left(x^{k} , y^{k}\right) \leq \frac{\delta - \varepsilon}{\left(1 + {\underline L_{k}\tau_{k - 1}}\right)}D_{h}\left(x^{k - 1} , x^{k}\right) \leq \frac{\delta - \varepsilon}{\left(1 + {\underline L_{k}\tau_{k}}\right)}D_{h}\left(x^{k-1} , x^{k}\right),
		\end{equation*}
		where the first inequality is due to the step \eqref{CoCaIn:3} of the algorithm and the second inequality is due to fact that $\tau_{k} \leq \tau_{k - 1}$. By rearranging we obtain,
		\begin{equation*}
			\delta D_{h}\left(x^{k-1} , x^{k}\right) - \left(1 + {\underline L_{k}\tau_{k}}\right)D_{h}\left(x^{k} , y^{k}\right) \geq \varepsilon D_{h}\left(x^{k-1} , x^{k}\right)
		\end{equation*}
		thus completing the proof.
	\end{proof}
	\begin{proposition} \label{P:SuffDesc0-1}
		Let $\Seq{x}{k}$ be a sequence generated by CoCaIn BPG. Then, the following assertions hold:
		\begin{itemize}
			\item[$\rm{(i)}$] The sequence $\left\{ \Phi_{\delta}^{k + 1}\left(x^{k + 1} , x^{k}\right) \right\}_{k \in \nn}$ is nonincreasing.
			\item[$\rm{(ii)}$] $\sum_{k = 1}^{\infty} D_{h}\left(x^{k - 1} , x^{k}\right) < \infty$, and hence the sequence $\left\{ D_{h}\left(x^{k - 1} , x^{k}\right)  \right\}_{k \in \nn}$ converges to zero.
			\item[$\rm{(iii)}$] $\min_{1 \leq k \leq n} D_{h}\left(x^{k - 1} , x^{k}\right) \leq\Phi_{\delta}^{1}\left(x^{1} , x^{0}\right)/\left(\varepsilon n\right)$.
		\end{itemize}    		
	\end{proposition}
	\begin{proof}
		\begin{itemize}
			\item[$\rm{(i)}$]  This follows trivially from Proposition \ref{prop:Lyapunov-1}, since $\varepsilon >0$.
			\item[$\rm{(ii)}$] Let $n$ be a positive integer. Summing \eqref{prop:Lyapunov-1:1} from $k = 1$ to $n$ we get
				\begin{equation} \label{P:SuffDesc0:2}
					\sum_{k = 1}^{n} D_{h}\left(x^{k - 1} , x^{k}\right) \leq \frac{1}{\varepsilon}\left(\Phi_{\delta}^{1}\left(x^{1} , x^{0}\right) - \Phi_{\delta}^{n +1 }\left(x^{n + 1} , x^{n}\right)\right) \leq \frac {1}{\varepsilon}\Phi_{\delta}^{1}\left(x^{1} , x^{0}\right), 
				\end{equation}
				 since $\Phi_{\delta}^{n + 1}\left(x^{n + 1} , x^{n}\right) \geq 0$. Taking the limit as $n \rightarrow \infty$, we obtain the first desired assertion, from which  we immediately deduce that $\left\{ D_{h}\left(x^{k - 1} , x^{k}\right) \right\}_{k \in \nn}$ converges to zero.
			\item[$\rm{(iii)}$] From \eqref{P:SuffDesc0:2} we also obtain,
				\begin{equation*}
					n\min_{1\leq k \leq n} D_{h}\left(x^{k - 1} , x^{k}\right) \leq \sum_{k = 1}^{n} D_{h}\left(x^{k - 1} , x^{k}\right) \leq \frac{1}{\varepsilon}\Phi_{\delta}^{1}\left(x^{1} , x^{0}\right),
				\end{equation*}
				which after division by $n$ yields the desired result.
		\end{itemize}
		\vspace{-0.2in}
	\end{proof}
	In order to proceed with the global convergence analysis of CoCaIn BPG, we will need throughout the rest of this section, to additionally assume the following.
	\begin{assumption} \label{A:AssumptionD}
		\begin{itemize}
			\item[$\rm{(i)}$] $\dom h = \real^{d}$.
			\item[$\rm{(ii)}$] $\nabla h$ and $\nabla g$ are Lipschitz continuous on any bounded subset of $\real^{d}$.
		\end{itemize}
	\end{assumption}
	
\subsection{Global Convergence for CoCaIn BPG}\label{SSec:global-convergence}
	In this subsection we show the global convergence result of CoCaIn BPG. The goal is to show that the whole sequence $\Seq{x}{k}$, that is generated by CoCaIn BPG, converges to a critical point. To this end, we denote the set of critical points by 
	\begin{equation*}
		\crit \Psi = \left\{ x \in \real^{d} : \;  0 \in \partial \Psi\left(x\right) \equiv \partial f\left(x\right) + \nabla g\left(x\right) \right\}.
	\end{equation*}
	Note that, such a set is well-defined due to Fermat's rule \cite[Theorem 10.1, p. 422]{RW1998-B} and due to the concept of limiting subdifferential.  
\medskip
	
	From now on we will make the following assumption regarding the sequence of majorant parameters $\left\{ \uL_{k} \right\}_{k \in \nn}$: there exists an integer $K \in \nn$ such that $\uL_{k} = \uL$ for all $k \geq K$ ($K$ can be as large as the user wishes). It should be noted that thanks to Assumption \ref{A:AssumptionC}(ii) and Lemma \ref{L:ExtenDL}, there exists a global majorant parameter $\uL$ such that \eqref{CoCaIn:5} holds true for all $k \in \nn$. On the other hand, since in anyway we require that the parameters do not decrease between two successive iterations, it makes sense that at some point we will stop changing them and continue with a fixed value. However, it is very important not using the global parameter $\uL$ right from the beginning since in practice the parameter $\uL_{k}$ determined by \eqref{CoCaIn:5} might be much smaller (especially in early stages of the algorithm).
\medskip

	In the second phase of the algorithm, \ie when $k \geq K$, it also makes sense to assume that $\tau_{k} = \tau$ for all $k \geq K$ where $\tau \leq \uL^{-1}$. This immediately suggests that our Lyapunov function can also be simplified. More precisely, we define the following new Lyapunov function: 
	\begin{equation} \label{eq:new-lyapunov}
		\Psi_{\delta_{1}}\left(x , y\right) = 
		\begin{cases}  
			\Phi_{\delta}^{k}\left(x , y\right), & x = x^{k}, \, y = x^{k - 1}, \text{for some} \, k < K, \\
			\Psi\left(x\right) + \delta_{1}D_{h}\left(y , x\right), & \text{otherwise},
		\end{cases}
	\end{equation}
	where $\delta_{1} = \delta/\tau$.
\medskip

	The global convergence result is based on showing that CoCaIn BPG generates a gradient-like descent sequence according to Definition \ref{D:GradientLDS} (see below). This involves three properties which need to be verified: ``sufficient descent condition'', ``relative error condition'' and ``continuity condition''.  Such a convergence analysis is based on a recent technique, which was initiated by Attouch and Bolte \cite{AB2009}, and later on was simplified and unified in \cite{BST2014}. A more general framework was proposed in \cite{Ochs2019}.
\medskip

	The main tool that stands behind this technique is the \textit{Kurdyka-{\L}ojasiewicz} (KL) property \cite{K1998,L1963} (see \cite{BDL2006} for the non-smooth case), which is properly defined in the appendix. This property has been used in several recent works that deal with non-convex optimization problems (see \cite{AB2009,ABS2013,BST2014} for early foundational works). For more details and information on the KL property, we refer the reader to the following papers \cite{BDL2006,AB2009,BDLM2010,ABRS2010,ABS2013,BST2014, Ochs2019} and references therein.
\medskip

	Verifying that a given function satisfies the KL property could be difficult, however in their seminal work \cite{BDL2006}, Bolte, Daniilidis and Lewis prove that any proper, lower semicontinuous and semi-algebraic function satisfies the KL property on its domain. This important result makes this proof technique very powerful, since we are familiar with many semi-algebraic functions that appear very often in applications. In fact, the same result holds for (possibly non-smooth) functions that are definable in an o-minimal structure \cite{BDL2006, BDLS2007}. For examples and more details about the relations between KL and other important notions, see \cite{BDL2006,BDLM2010} and references therein.
\medskip

	 In order to derive the global convergence of our algorithm we follow this proof technique that we shortly recall now. For the interested readers we refer to \cite[Appendix 6, p. 2147]{BSTV2018}, where a short and self-contained summary of this proof methodology can be found. It should be noted again that here we consider a modification, which fits non-descent methods like CoCaIn BPG.
	\begin{definition}[Gradient-like Descent Sequence] \label{D:GradientLDS}
		A sequence $\Seq{x}{k}$ is called \textit{a gradient-like descent sequence} for minimizing $\Psi_{\delta_{1}}$ if the following three conditions hold:
		\begin{itemize}
        		\item[$\rm{(C1)}$] \textit{Sufficient decrease condition.} There exists a positive scalar $\rho_{1}$ such that
            		\begin{equation*}
                		\rho_{1}\norm{x^{k} - x^{k - 1}}^{2} \leq \Psi_{\delta_1}\left(x^{k} , x^{k - 1}\right) - \Psi_{\delta_1}\left(x^{k + 1} , x^{k}\right), \quad \forall \,\, k \in \nn.
	            \end{equation*}
    			\item[$\rm{(C2)}$] \textit{Relative error condition.} There exist an integer $K \in \nn$ and a positive scalar $\rho_{2}$ such that         	
		   		\begin{equation*}
	    				\norm{w^{k + 1}} \leq \rho_{2}\left(\norm{x^{k} - x^{k - 1}} + \norm{x^{k + 1} - x^{k}}\right), \quad w^{k + 1} \in \partial \Psi_{\delta_1}\left(x^{k + 1} , x^{k}\right), \quad \forall \,\, k \geq K.
	            \end{equation*}
			\item[$\rm{(C3)}$]  \textit{Continuity condition.} Let $\overline{x}$ be a limit point of a subsequence $\left\{ x^{k} \right\}_{k \in {\cal K}}$, then $\limsup_{k \in {\cal K} \subset \nn} \Psi\left(x^{k}\right) \leq \Psi\left(\overline{x}\right)$.
        \end{itemize}	
	\end{definition}	
	Based on Definition \ref{D:GradientLDS} and the KL property, the following global convergence result holds true. 	We provide its proof in the appendix. 
	\begin{theorem}[Global Convergence] \label{T:AbstrGlob}
		Let $\Seq{x}{k}$ be a bounded gradient-like descent sequence for minimizing $\Psi_{\delta_{1}}$. If $\Psi$ satisfies the KL property, then the sequence $\Seq{x}{k}$ has finite length, \ie $\sum_{k = 1}^{\infty} \norm{x^{k + 1} - x^{k}} < \infty$ and it converges to $x^{\ast} \in \crit \Psi$.
	\end{theorem}
	Now, in a sequence of lemmas, we prove that CoCaIn BPG generates a gradient-like descent sequence for minimizing $\Psi_{\delta_{1}}$. In order to prove condition (C1), we first note that Proposition \ref{P:SuffDesc0-1} is also valid  for the new Lyapunov function $\Psi_{\delta_{1}}$ as recorded now (for the sake of simplicity we omit the exact details of the proof, which is almost identical to the proof above).
	\begin{proposition} \label{P:SuffDesc}
		Let $\Seq{x}{k}$ be a sequence generated by CoCaIn BPG. Then, the following assertions hold:
		\begin{itemize}
			\item[$\rm{(i)}$] The sequence $\left\{ \Psi_{\delta_{1}}\left(x^{k + 1} , x^{k}\right) \right\}_{k \in \nn}$ is nonincreasing, converging and condition (C1) of Definition \ref{D:GradientLDS} holds true. 
			\item[$\rm{(ii)}$] $\sum_{k = 1}^{\infty} D_{h}\left(x^{k - 1} , x^{k}\right) < \infty$, and hence the sequence $\left\{ D_{h}\left(x^{k - 1} , x^{k}\right)  \right\}_{k \in \nn}$ converges to zero.
			\item[$\rm{(iii)}$] $\min_{1 \leq k \leq n} D_{h}\left(x^{k - 1} , x^{k}\right) \leq \left(\Psi_{\delta_{1}}\left(x^{1} , x^{0}\right) - \Psi_{\ast}\right)/\left(\varepsilon n\right)$ where $\Psi_{\ast} = v(\PPP) > -\infty$ (by Assumption \ref{A:AssumptionA}(iv)).
		\end{itemize}
	\end{proposition}
	% \MMC[inline]{Bug in Part [$\rm{(i)}$ as we remove $\tau$ in the second part of the definition of Lyapunov function~\eqref{eq:new-lyapunov}. Also, check for any changes in Part $\rm{(iii)}$, some inconsistency might be there with $v(\PPP)$. 

	% An easy fix is to remove the first part of the \eqref{eq:new-lyapunov} and explicitly tell that the further analysis is valid for $k > K$. What do you think?.
	% }
	Now we can prove the following result, which means that condition (C2) holds true.
	\begin{proposition} \label{P:SubGB}
		Let $\Seq{x}{k}$ be a bounded sequence generated by CoCaIn BPG. Then, there exist $w^{k + 1} \in \partial \Psi_{\delta_{1}}\left(x^{k + 1} , x^{k}\right)$ and a positive scalar $\rho_{2}$ such that
		\begin{equation*}
	   		\norm{w^{k + 1}} \leq \rho_{2}\left(\norm{x^{k} - x^{k - 1}} + \norm{x^{k + 1} - x^{k}}\right), \quad \forall \,\, k \geq K.
		\end{equation*}
	\end{proposition}
	\begin{proof}
		Fix $k \geq K$. By the definition of the Lyapunov function $\Psi_{\delta_1}\left(\cdot , \cdot\right)$ we obtain that
		\begin{equation*}
			\partial \Psi_{\delta_1}\left(x^{k + 1} , x^{k}\right) = \left(\partial \Psi\left(x^{k + 1}\right) + \delta_1 \nabla^{2} h\left(x^{k + 1}\right)\left(x^{k + 1} - x^{k}\right) , \delta_1 \Big( \nabla h\left(x^{k}\right) -  \nabla h\left(x^{k + 1}\right)\Big)\right).
		\end{equation*}
		Writing the optimality condition of the optimization problem which defines $x^{k + 1}$ (see \eqref{CoCaIn:2} and recall that for $k \geq K$, we have that $\tau_{k} = \tau$) yields that
		\begin{equation*}
			0 \in \partial f\left(x^{k + 1}\right) + \nabla g\left(y^{k}\right) + \frac{1}{\tau}\left(\nabla h\left(x^{k + 1}\right) - \nabla h\left(y^{k}\right)\right).
		\end{equation*}
		Therefore
		\begin{equation*}
			\nabla g\left(x^{k + 1}\right) - \nabla g\left(y^{k}\right) + \frac{1}{\tau}\left(\nabla h\left(y^{k}\right) - \nabla h\left(x^{k + 1}\right)\right) \in \partial\Psi\left(x^{k + 1}\right),
		\end{equation*}		
		and by defining
		\begin{equation*}
			w_{1}^{k + 1} \equiv \nabla g\left(x^{k + 1}\right) - \nabla g\left(y^{k}\right) + \frac{1}{\tau}\left(\nabla h\left(y^{k}\right) - \nabla h\left(x^{k + 1}\right)\right) + \delta_{1}\nabla^{2} h\left(x^{k + 1}\right)\left(x^{k + 1} - x^{k}\right),
		\end{equation*}
		and $w_{2}^{k + 1} \equiv \delta_{1}\left(\nabla h\left(x^{k}\right) -  \nabla h\left(x^{k + 1}\right)\right)$ we obviously obtain that $w^{k + 1} \in \partial \Psi_{\delta_{1}}\left(x^{k + 1} , x^{k}\right)$ where $w^{k + 1} = \left(w_{1}^{k + 1} , w_{2}^{k + 1}\right)$. Since $\Seq{x}{k}$ is a bounded sequence and both $\nabla h$ and $\nabla g$ are Lipschitz continuous on bounded subsets of $\real^{d}$ (see Assumption \ref{A:AssumptionD}(ii)), there exists $M > 0$ such that
		\begin{align*}
			\norm{w_{1}^{k + 1}} & \leq \norm{\nabla g\left(x^{k + 1}\right) - \nabla g\left(y^{k}\right)} + \frac{1}{\tau}\norm{\nabla h\left(y^{k}\right) - \nabla h\left(x^{k + 1}\right)} + \delta_{1}\norm{\nabla^{2} h\left(x^{k + 1}\right)}\cdot\norm{x^{k + 1} - x^{k}} \\
			& \leq M\left(1 + \frac{1}{\tau}\right)\norm{x^{k + 1} - y^{k}} + \delta_1 M\norm{x^{k + 1} - x^{k}},
		\end{align*}
		where the last inequality follows also from the fact that $\norm{\nabla^{2} h\left(x^{k + 1}\right)} \leq M$, since $\nabla h$ is Lipschitz continuous on bounded subsets of $\real^{d}$. Using step \eqref{CoCaIn:1} we obtain that
		\begin{align*}
			\norm{w_{1}^{k + 1}} & \leq M\left(1 + \frac{1}{\tau}\right)\left(\norm{x^{k + 1} - x^{k}} + \gamma_{k}\norm{x^{k} - x^{k - 1}}\right) +  \delta_{1}M\norm{x^{k + 1} - x^{k}} \\
			& \leq M\left(1 + \delta_{1} + \frac{1}{\tau}\right)\norm{x^{k + 1} - x^{k}} + M\left(1 + \frac{1}{\tau}\right)\norm{x^{k} - x^{k - 1}},
		\end{align*}
		where we have used the fact that $\gamma_{k} \leq 1$, $k \in \nn$. 	Since, we also have that 
		\begin{equation*}
			\norm{w_{2}^{k + 1}} =  \delta_1 \norm{ \nabla h\left(x^{k}\right) -  \nabla h\left(x^{k + 1}\right)}  \leq  \delta_1 M\norm{x^{k + 1} - x^{k}},
		\end{equation*}
		the desired result is proved and condition (C2) also holds true.
	\end{proof}
	Now we are left with showing that CoCaIn BPG generates a sequence that satisfies condition (C3).
	\begin{proposition}\label{prop:subsequence0}
		Let $\Seq{x}{k}$ be a bounded sequence generated by CoCaIn BPG. Let $x^{\ast}$ be a limit point of a subsequence $\left\{ x^{k} \right\}_{k \in {\cal K}}$, then $\limsup_{k \in {\cal K} \subset \nn} \Psi\left(x^{k}\right) \leq \Psi\left(x^{\ast}\right)$.
	\end{proposition}
	\begin{proof}	
		Consider a subsequence $\left\{ x^{n_{k}} \right\}_{k \in \nn}$ which converges to $x^{\ast}$ (there exists such a subsequence since the sequence $\Seq{x}{k}$ is assumed to be bounded). Using Proposition \ref{P:SuffDesc}(ii) and the strong convexity of $h\left(\cdot\right)$, we obtain that $\lim_{k \rightarrow \infty} \norm{x^{k} - x^{k - 1}} = 0$. Therefore, the sequence $\left\{ x^{n_k - 1} \right\}_{k \in \nn}$ also converges to $x^{\ast}$. From the definition of $y^{k}$, see \eqref{CoCaIn:1}, it also follows that $\left\{ y^{n_k - 1} \right\}_{k \in \nn}$ also converges to $x^{\ast}$. In addition, since $h$ is continuously differentiable on $\real^{d}$ we have that $\lim_{k \rightarrow \infty} D_{h}\left(x^{\ast} , y^{{n_k} - 1}\right) = 0$.  Now, from \eqref{CoCaIn:2}, it follows (after some simplifications), for all $k \geq K$, that
		\begin{equation*}
			f\left(x^{k}\right) \leq f\left(x^{\ast}\right) + \act{x^{\ast} - x^{k} , \nabla g\left(y^{k - 1}\right)} + \frac{1}{\tau}D_{h}\left(x^{\ast} , y^{k - 1}\right) - \frac{1}{\tau}D_{h}\left(x^{k} , y^{k - 1}\right).
		\end{equation*}
		Substituting $k$ by $n_{k}$ and letting $k \rightarrow \infty$, we obtain from the fact that $g$ is continuously differentiable on $\real^{d}$, that
		\begin{equation*}
			\limsup_{k \rightarrow \infty} f\left(x^{n_{k}}\right) \leq f\left(x^{\ast}\right).
		\end{equation*}
		Using this, and recalling that here $g$ is continuous, we obtain that $\limsup_{k \in {\cal K} \subset \nn} \Psi\left(x^{n_{k}}\right) \leq \Psi\left(x^{\ast}\right)$, where ${\cal K} = \left\{ n_{k} :  \, k \geq K \right\}$.
	\end{proof}
	The global convergence of CoCaIn BPG now easily follows from our general result on gradient-like descent sequences (see Theorem \ref{T:AbstrGlob})		
	\begin{theorem}[Global Convergence of CoCaIn BPG] \label{T:GlobalCoCaInBPG}
		Let $\Seq{x}{k}$ be a bounded sequence generated by CoCaIn BPG. If $f$ and $g$ satisfy the KL property, then the sequence $\Seq{x}{k}$ has finite length, \ie $\sum_{k = 1}^{\infty} \norm{x^{k + 1} - x^{k}} < \infty$ and it converges to $x^{\ast} \in \crit \Psi$.
	\end{theorem}
	Before we conclude this section, we provide a simplified variant of CoCaIn BPG.

\subsection{CoCaIn BPG Without Backtracking} \label{SSec:without-backtracking}
	Note that CoCaIn BPG uses a local estimate of the minorant and majorant parameters $\lL_{k}$ and $\uL_{k}$, $k \in \nn$, determined by the backtracking steps \eqref{CoCaIn:4} and \eqref{CoCaIn:5}, respectively. However, when the global parameter $L$ is known (guaranteed in Assumption \ref{A:AssumptionC}(ii)), we can skip the backtracking steps, and provide a simplified variant of CoCaIn BPG.
	
	\begin{figure}[h]
		\centering
		\fcolorbox{black}{Orange!10}{\parbox{15cm}{{\bf CoCaIn BPG Without Backtracking} \\
			{\bf Input.} $\delta, \varepsilon > 0$ with $1 > \delta > \varepsilon$. \\
			{\bf Initialization.} $x^{0} = x^{1} \in \idom h \cap \dom f$, $L \geq \max\{\frac{-\alpha}{(1 - \delta)\sigma}, L \}$ and $\tau_{0}\leq L^{-1}$. \\
			{\bf General Step.} For $k = 1 , 2 , \ldots$, compute
			\begin{align}
				y^{k	} & = x^{k} + \gamma_{k}\left(x^{k} - x^{k - 1}\right) \in \idom h, \label{CoCaIn0:21} \\
				x^{k + 1} & \in \argmin_{u} \left\{ f\left(u\right) + \act{\nabla g\left(y^{k}\right) , u - y^{k}} + \frac{1}{\tau_{k}}D_{h}\left(u , y^{k}\right) \right\}, \label{CoCaIn0:22}
			\end{align}
			where $\tau_{k} \leq \min\{\tau_{k-1}, L^{-1}\}$ and $\gamma_{k} \geq 0$ satisfies
			\begin{equation} \label{CoCaIn0:23}
				\left(\delta - \varepsilon\right)D_{h}\left(x^{k - 1} , x^{k}\right) \geq 2D_{h}\left(x^{k} , y^{k}\right)\,.
			\end{equation}}}
	\end{figure}
	\medskip 

	For the inertial step \eqref{CoCaIn0:23}, when $h = \left(1/2\right)\norm{\cdot}^{2}$ we can obtain that 
	\begin{equation*}
		\gamma_{k} \leq \sqrt{\frac{\delta - \epsilon}{2}},
	\end{equation*}
	with $\uL = \lL$. Using Remark \ref{rem:inertia-rem}, if $\delta - \epsilon \approx 1$, one could choose the extrapolation parameter as follows $\gamma_{k} \approx 1/\sqrt{2}$. However, in general, the closed form expression for $\gamma_{k}$ is difficult to obtain, for which backtracking line-search strategy can be used. Recently, in \cite{MWLCO2019} the authors showed a technique to obtain closed form inertia for general Bregman distances. We use their technique later in the context of Quadratic inverse problems to propose a new variant of CoCaIn BPG with closed form inertia.

\subsection{Implementing the Double Backtracking Procedure} \label{SSec:Implement} 

	The update steps of CoCaIn BPG are based on the double backtracking strategy (see steps \eqref{CoCaIn:4} and \eqref{CoCaIn:5}). Here, we describe some implementation details of these two steps. Note that the inner loops for finding the minorant and the majorant parameters $\lL_{k}$ and $\uL_{k}$, $k \in \nn$, are implemented in a sequential fashion. By this, we mean that at iteration $k \in \nn$ we first execute the steps \eqref{CoCaIn:1}, \eqref{CoCaIn:3} and \eqref{CoCaIn:4} in order to compute an appropriate $y^{k}$, only then we proceed to steps \eqref{CoCaIn:2} and \eqref{CoCaIn:5} in order to compute $x^{k + 1}$.  Note that the fact that the sequence $\left\{ \uL_{k} \right\}_{k \in \nn}$ does not decrease is crucial in order to decouple the steps \eqref{CoCaIn:1} and \eqref{CoCaIn:2}. More precisely, we now describe the backtracking procedure to find $\lL_{k}$. Let $\underline{\nu} > 1$ be a scaling parameter and arbitrarily initialize $\lL_{k,0} > 0$. Then, we find the smallest $\lL_{k} \in \left\{ \underline{\nu}^{0}\lL_{k,0} , \underline{\nu}^{1}\lL_{k,0} , \underline{\nu}^{2}\lL_{k,0} , \ldots \right\}$ that satisfies \eqref{CoCaIn:4} and such that $\gamma_{k} \geq 0$ satisfies
	\begin{equation*}
		D_{h}\left(x^{k} , y^{k}\right) \leq \frac{\delta - \varepsilon}{\lL_{k}\tau_{k - 1} + 1}D_{h}\left(x^{k - 1} , x^{k}\right).
	\end{equation*}
	We can now describe the procedure to find $\uL_{k}$. Let $\overline{\nu} > 1$ and initialize $\uL_{k,0} := \uL_{k - 1}$, then we take the smallest $\uL_{k} \in \left\{ \overline{\nu}^{0}\uL_{k,0} , \overline{\nu}^{1}\uL_{k,0} , \overline{\nu}^{2}\uL_{k,0} , \ldots \right\}$ that satisfies \eqref{CoCaIn:5}. Therefore, $\left\{ \uL_{k} \right\}_{k \in \nn}$ is monotonically non-decreasing. Note, however, we do not require any monotonicity of the sequence $\left\{ \lL_{k} \right\}_{k \in \nn}$. 
\medskip

	The double backtracking strategy preserves the sign of $\lL_{k}$, however, only $-\lL_{k} \leq \uL_{k}$ is required. Changing the sign of $\lL_{k}$ when the function is locally strongly convex might lead to additional acceleration. However, we leave this kind of adaptation for future work. 

\section{Numerical Experiments} \label{Sec:Numerics}
	Our goal in this section is to illustrate the performance of CoCaIn BPG in various situations. We start with minimization of univariate functions, which emphasizes the power of incorporating inertial terms into the BPG algorithm and using the double backtracking procedure. Then we provide some insights on the following practical applications: Quadratic Inverse Problems in Phase Retrieval and Non-convex Robust Denoising with Non-convex Total Variation Regularization. More recently, the efficiency of  CoCaIn BPG is also demonstrated in related work for Matrix Factorization \cite{MO2019a} and Deep Linear Neural Networks \cite{MWLCO2019}.
\subsection{Finding Global Minima of Univariate Functions}
We begin with two examples of minimizing univariate non-convex functions, which shed some light on the two main features of our algorithm: (i) inertial term and (ii) double backtracking procedure. We consider unconstrained minimization of functions $g : \real \rightarrow \real$, with Lipschitz continuous gradient, \ie model ($\PPP$) with $d = 1$, $f \equiv 0$ and $C = \real$. The two functions are: $g\left(x\right) = \log\left(1 + x^{2}\right)$ and $g\left(x\right) = \left(1 + e^{x}\right)^{-1}$. We compare three methods: CoCaIn BPG with $h = \left(1/2\right)\norm{\cdot}^{2}$ and refer to it as \textit{CoCaIn with Euclidean distance},  classical \textit{Gradient Descent} (GD) method with backtracking (which is actually CoCaIn with Euclidean distance and with $\gamma_{k} = 0$ for all $k \in \nn$), and \textit{iPiano}\footnote{In this particular case, the method coincides with the Heavy-ball method \cite{P1964}.} \cite{OCBP2014} (with the inertial parameter set to $0.7$). When using a backtracking procedure in GD and iPiano methods, we mean that only the majorant parameter is varied. We use the same initialization for all the algorithms and report the performance in Figure \ref{fig:simple_functions}. 

	\begin{figure}[htb]
		\centering
		\begin{subfigure}{0.48\textwidth}
			\centering
			\includegraphics[width=0.9\textwidth]{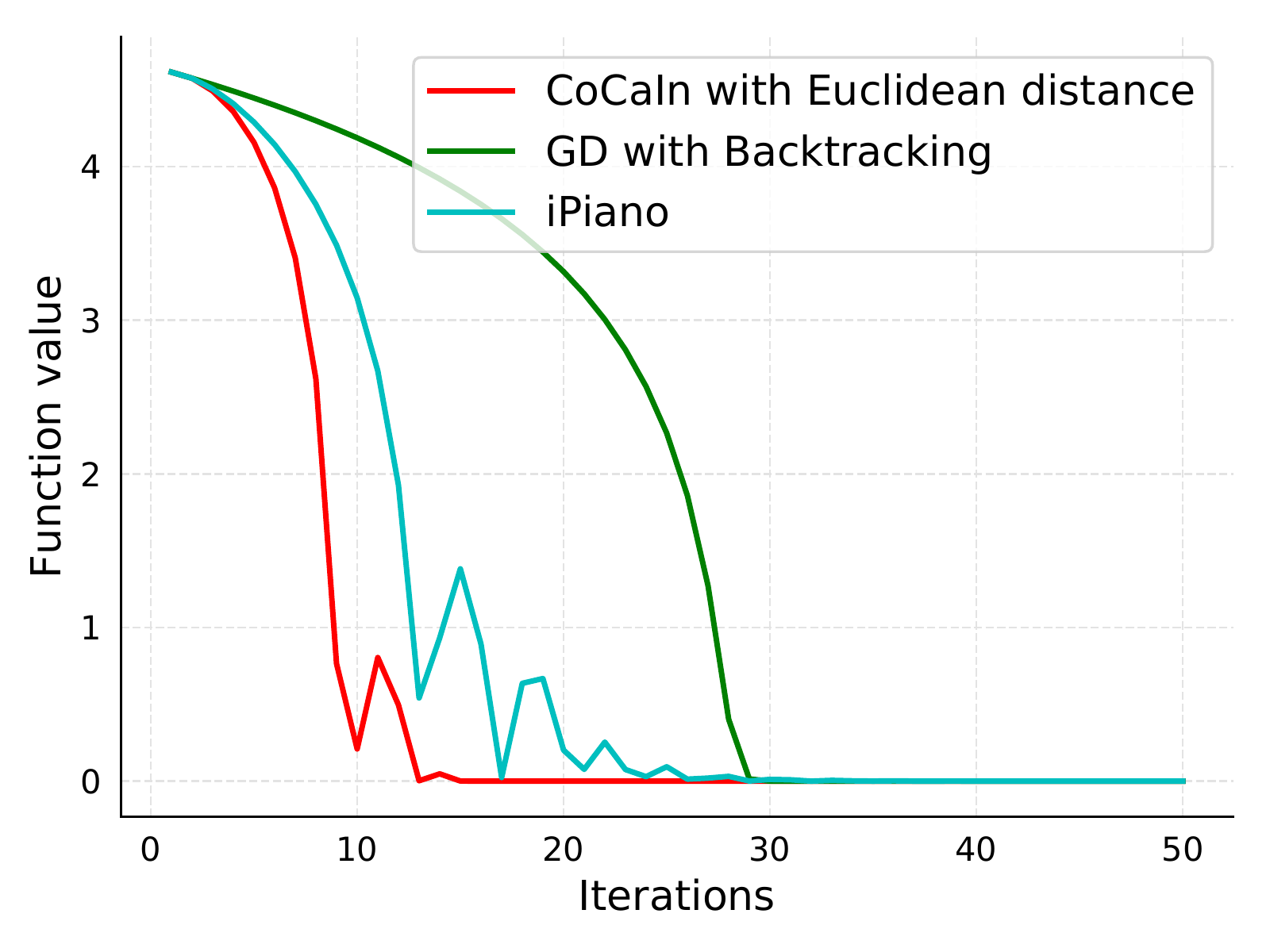}
			\vspace{0.2em}
			\caption{$g\left(x\right) = \log\left(1 + x^{2}\right)$}
		\end{subfigure}
		\hfill
		\begin{subfigure}{0.48\textwidth}
			\centering
			\includegraphics[width=0.9\textwidth]{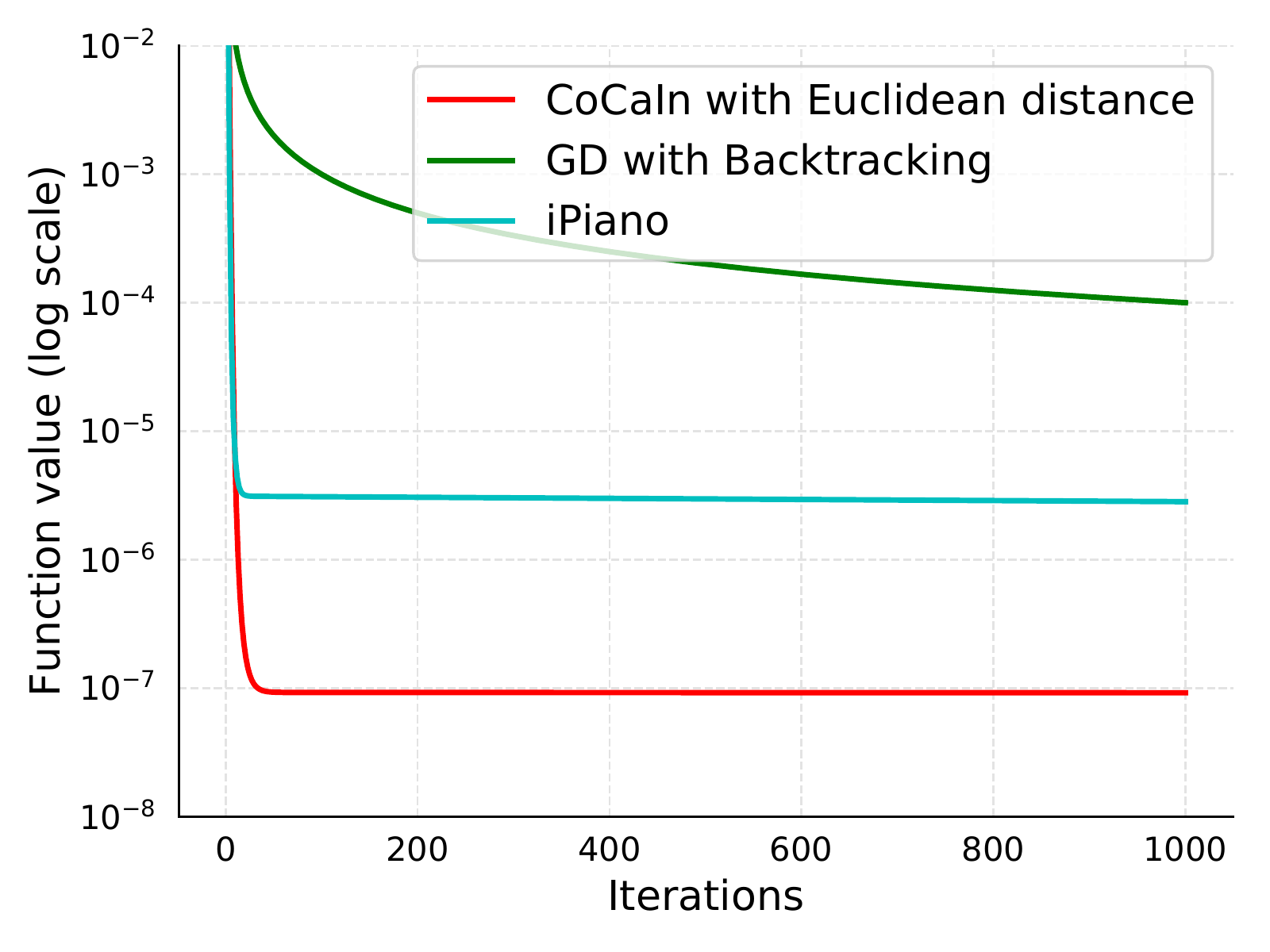}
			\vspace{0.2em}
			\caption{$g\left(x\right) = \frac{1}{1 + e^{x}}$}
		\end{subfigure}
		\caption{\textbf{Better performance by CoCaIn.} In the left-hand side plot, the function has a unique critical point. CoCaIn BPG finds it faster than the other two methods. In the right-hand side plot, the function has a very small gradient and CoCaIn BPG reaches a significantly lower function value  than the two other methods. These plots hint that CoCaIn BPG can  significantly accelerate the convergence speed with comparison to GD and iPiano which use only  a simple backtracking procedure.}
		\label{fig:simple_functions}
	\end{figure}  
	
	In the second experiment, we illustrate the robustness of CoCaIn BPG to local minima and critical points. We consider the non-smooth and non-convex function $\Psi\left(x\right) = \left|x\right| + \sin\left(x\right) + \cos\left(x\right)$, with many critical points as shown in the center plot of Figure \ref{fig:simple_functions_1}, and set $f\left(x\right) = \left|x\right|$ and $g\left(x\right) = \sin\left(x\right) + \cos\left(x\right)$ (which is obviously a non-convex function with Lipschitz continuous gradient). Here again we take $h = \left(1/2\right)\norm{\cdot}^{2}$. In order to apply CoCaIn BPG, the main computational step is of the following form:
	\begin{equation}\label{eq:forward-backward}
 		x^{k + 1} \in \argmin_{x} \left\{ \left|x\right| + \act{x - y^{k} , \cos\left(y^{k}\right) - \sin\left(y^{k}\right)} + \frac{1}{2\tau_{k}}\left(x - y^{k}\right)^{2} \right\},
	\end{equation}
	which results in the following update step
	\begin{equation}
		x^{k + 1} = \max\left\{ 0 ,\left|y^{k} - \tau_{k}\nabla g\left(y^{k}\right)\right| - \tau_{k} \right\}\sgn\left(y^{k} - \tau_{k}\nabla g\left(y^{k}\right)\right).	
	\end{equation}
	We compare CoCaIn BPG with Euclidean distance to the classical \textit{Proximal Gradient} (PG) method with backtracking (CoCaIn BPG with  Euclidean distance and $\gamma_{k} = 0$, $k \in \nn$), and \textit{iPiano}. As mentioned in the first experiment, when using a backtracking procedure in PG and iPiano methods we mean that only the majorant parameter is varied.
        %Therefore, for both algorithms we set $\lL_{k} = 0$, $k \in \nn$.

	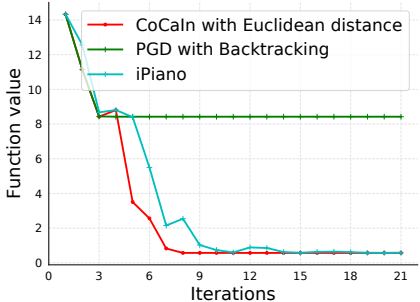
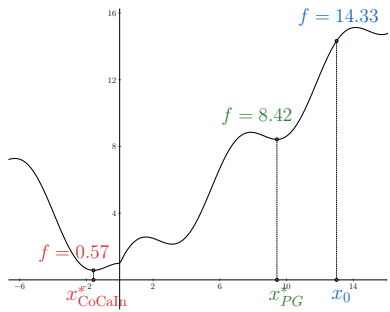
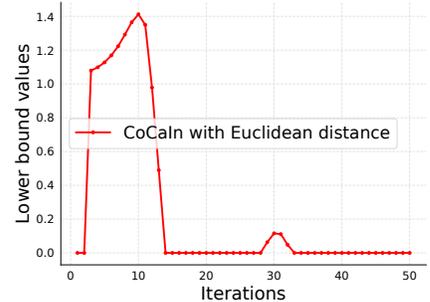
\begin{figure}[htb]
		\centering
		\begin{subfigure}[b]{0.32\textwidth}
			\centering
			\includegraphics[width=1\textwidth]{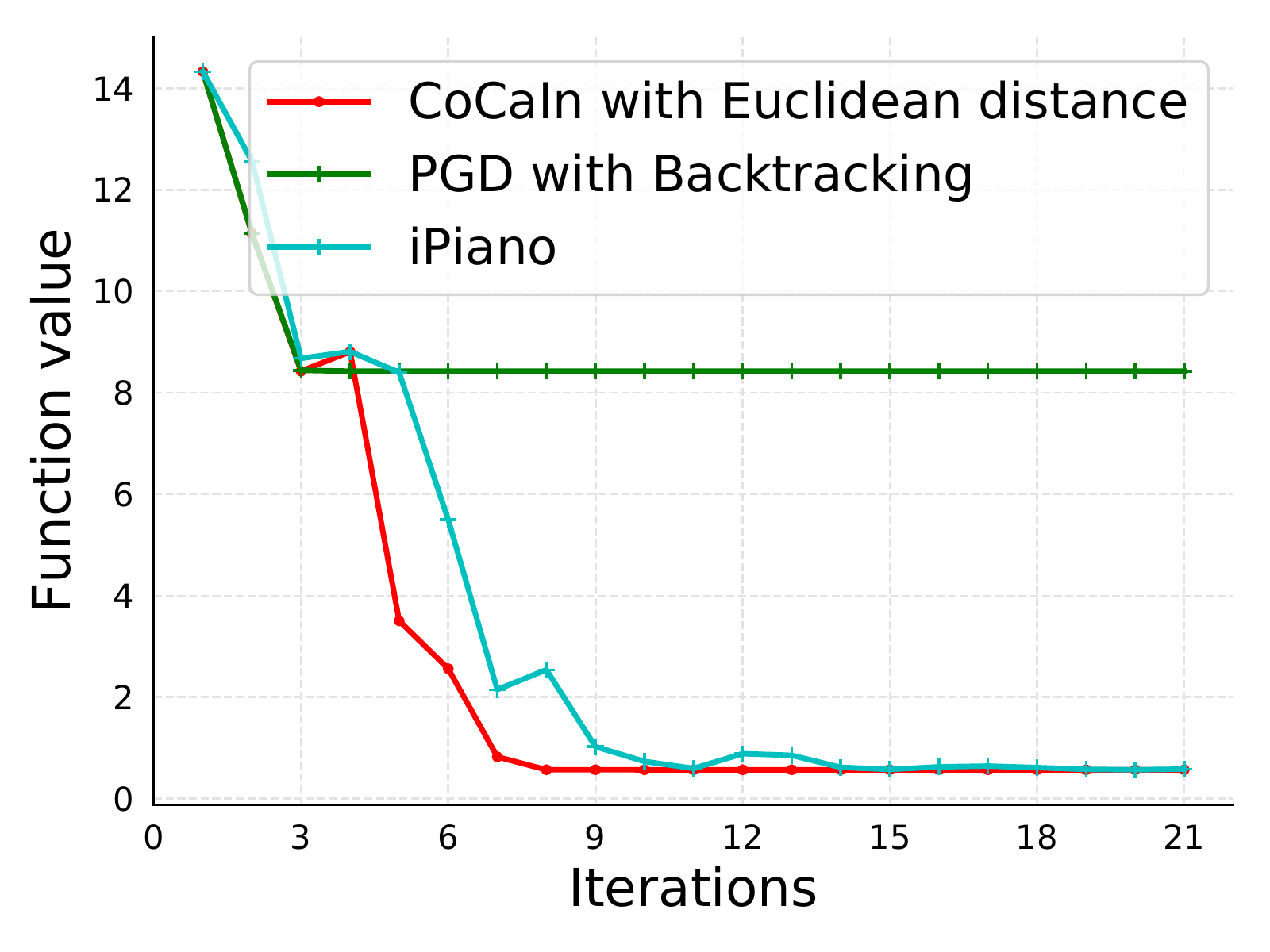}
			\caption{Function value plot}
		\end{subfigure}
		\begin{subfigure}[b]{0.32\textwidth}
			\centering
			\newcommand{\thisTikzScaling}{1.5}
			\pgfplotsset{compat=1.13}
\definecolor{wrwrwr}{rgb}{0.3803921568627451,0.3803921568627451,0.3803921568627451}
\definecolor{rvwvcq}{rgb}{0.08235294117647059,0.396078431372549,0.7529411764705882}
\definecolor{sexdts}{rgb}{0.1803921568627451,0.49019607843137253,0.19607843137254902}
\definecolor{dtsfsf}{rgb}{0.8274509803921568,0.1843137254901961,0.1843137254901961}
\resizebox{0.88\textwidth}{!}{
\begin{tikzpicture}[line cap=round,line join=round,>=triangle 45,x=1cm,y=1cm,scale=\thisTikzScaling]
\begin{axis}[
x=1cm,y=1cm,
axis lines=middle,
xmin=-6.647900342536912,
xmax=16.05925688995399,
ymin=-1.750748909987878,
ymax=16.29800173658318,
xtick={-6,-2,...,22},
ytick={-4,0,...,16},]
 \tikzstyle{every node}=[font=\Large]
\clip(-6.647900342536912,-5.150748909987878) rectangle (23.35925688995399,17.29800173658318);
\draw[line width=1.8pt,smooth,samples=100,domain=-6.647900342536912:23.35925688995399] plot(\x,{abs((\x))+sin(((\x))*180/pi)+cos(((\x))*180/pi)});
\draw [line width=1.2pt,dotted] (-1.57,0)-- (-1.57,0.5707966437788987);
\draw [line width=1.2pt,dotted] (9.42477796076938,0)-- (9.42477796076938,8.42477796076938);
\draw [line width=1.2pt,dotted] (13,0)-- (13,14.327613818276836);
\begin{scriptsize}
\draw [fill=dtsfsf] (-1.57,0) circle (2.5pt);
\draw[color=dtsfsf] (-1.36338250705296815,-0.97) node[scale=2.5] {$x_{\text{CoCaIn}}^*$};
\draw [fill=dtsfsf] (-1.57,0.5707966437788987) circle (2.5pt);
\draw[color=dtsfsf] (-2.742641674010258,1.6177750783368223) node[scale=2.3]  {$f = 0.57$};
\draw [fill=sexdts] (9.42477796076938,0) circle (2.5pt);
\draw[color=sexdts] (9.997579329825966,-0.97) node[scale=2.5]  {$x_{PG}^*$};
\draw [fill=sexdts] (9.42477796076938,8.42477796076938) circle (2.5pt);
\draw[color=sexdts] (8.233246079564404,9.829119886633935) node[scale=2.3] {$f=8.42$};
\draw [fill=rvwvcq] (13,0) circle (2.5pt);
\draw[color=rvwvcq] (13.281381442060818,-0.97) node[scale=2.5]  {$x_0$};
\draw [fill=rvwvcq] (13,14.327613818276836) circle (2.5pt);
\draw[color=rvwvcq] (13.11152960866936,15.730551333132124) node[scale=2.3] {$f = 14.33$};
\end{scriptsize}
\end{axis}
\end{tikzpicture}
}
			\caption{$\Psi\left(x\right) = \left|x\right| + \sin\left(x\right) + \cos\left(x\right)$}     
		\end{subfigure}
		\begin{subfigure}[b]{0.32\textwidth}
			\centering
			\includegraphics[width=1\textwidth]{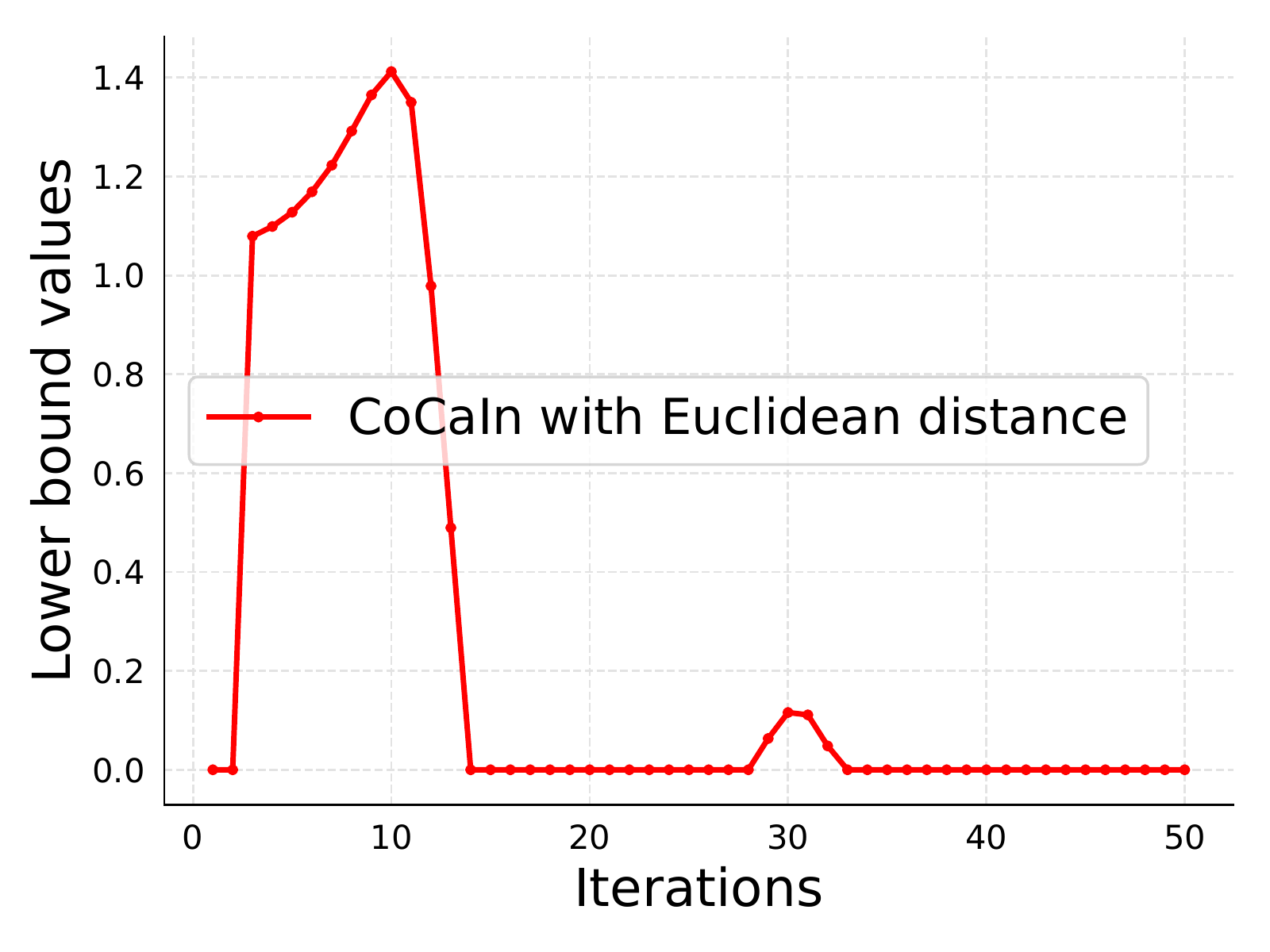}
			\caption{${\underline L_{k}}$ value}
			\label{fig:lower_bound}
		\end{subfigure}
		\caption {\textbf{CoCaIn can find the global minimum.} The left-hand side plot explicitly shows the behaviour in terms of function values versus the iterations counter. In the center plot, we use $x_{\text{PG}}^{\ast}$ as a short hand notation for the critical point achieved by the Proximal Gradient method with backtracking, and for CoCaIn BPG method we use $x_{\text{CoCaIn}}^{\ast}$. The iPiano method achieves the same critical point as the CoCaIn BPG method but slower. In the right-hand side plot, we plot $\lL_{k}$ (the minorant parameter) obtained by CoCaIn BPG method versus the iterations counter. The hilly structures represent that CoCaIn BPG can bypass local maxima and eventually converge to zero. Meaning that CoCaIn BPG adapts to the ``local convexity" of the function. } 
		\label{fig:simple_functions_1}
	\end{figure}
	
	As shown in Figure \ref{fig:simple_functions_1}, CoCaIn BPG achieves 
	the global minimum, whereas the PG with backtracking
	 gets stuck in a local minimum. We performed the same experiment 
	 starting at $100$ equidistant points sampled from the interval 
	 $\left[-15 , 15\right]$. The average final function value for 
	 CoCaIn was $2.75$, whereas for PG method 
	 with backtracking it was $3.21$ and for the iPiano it was $3.37$. 	 This means that CoCaIn BPG  reaches the global minimum from $52$ points, PG method with backtracking achieves 
	  the global minimum only from $27$ points and iPiano from $39$ 
	  points. Hence, the behavior illustrated in 
	  Figure~\ref{fig:simple_functions_1} is not due to the choice of 
	  initialization, but rather due to additional features of the CoCaIn 
	  BPG algorithm. 
	  This illustrates the great power of using double backtracking procedure 
	  in minimizing univariate non-convex functions.

\subsection{Escaping Spurious Stationary Points} \label{SSec:Escape}
	Here, we provide evidence that CoCaIn BPG can escape spurious stationary points in minimizing non-convex functions of two variables. Let $b_{i} \in \real$, $i = 1 , 2 , \ldots , m$, be samples of a noisy signal with additive Gaussian noise. A very common task in signal processing is to recover the true data. However, due to the noise, data can be prone to several outliers. In such cases, a robust loss \cite{FHT2001} is used. Moreover, prior information about the data, can be embedded through a regularizing term (for instance, a sparsity promoting regularizer). Given $\lambda , \rho > 0$, we consider minimization of
	\begin{equation} \label{eq:main-prob-1}
		\Psi\left(x\right) = \lambda \sum_{i = 1}^{m} \log\left(1 + \rho\left(x_{i} - b_{i}\right)^{2}\right) + \sum_{i = 1}^{m} \log\left(1 + \left|x_{i}\right|\right),
	\end{equation}
	with 
	\begin{equation*}
		f\left(x\right) :=  \sum_{i = 1}^{m} \log\left(1 + \left|x_{i}\right|\right) \quad \text{and} \quad g\left(x\right) := \lambda \sum_{i = 1}^{m} \log\left(1 + \rho\left(x_{i} - b_{i}\right)^{2}\right).
	\end{equation*}
	 The function $f$ is a non-convex sparsity promoting regularizer (also known as the log-sum penalty term \cite{CWB2008, N2005}) and the function $g$ is a robust loss. For illustration purposes, we consider a simple instance of problem \eqref{eq:main-prob-1} where $m = 2$, $\lambda = 0.5$ and $\rho = 100$. For minimizing this function we set $C = \real^{2}$ and $h\left(x\right) := \left(1/2\right)\left(x_{1}^{2} + x_{2}^{2}\right)$ to be used in the CoCaIn BPG method. 
\medskip

	Before presenting the numerical results, we would like to note that in this example, the function $f\left(x\right) - \left(\alpha/2\right)h\left(x\right)$ is convex for any $\alpha \leq -1$ and $Lh - g$ is convex for all $L \geq 100$. Each iteration of CoCaIn BPG would require to compute the Bregman proximal gradient mapping, which in this case reduces to the classical proximal gradient mapping (due to the choice of $h$). Note that due to the separability of the functions $f$ and $g$, the needed minimization problem can be split into two individual minimizations with respect to $x_{1}$ and $x_{2}$. These two optimization problems (after simple manipulations) reduces to  computation of the proximal mapping of the univariate function $\log\left(1 + \left|x\right|\right)$. A closed form formula can be found in \cite{GZZHY13} and reads as follows:
	\begin{align*}
		\prox_{\tau\log\left(1 + \left|x\right|\right)}\left(y\right) =
		\begin{cases}
			\sgn\left(y\right) \argmin_{x \in E} \left\{ \log\left(1 + \left|x\right|\right) +  \frac{1}{2\tau}\left(x - |y|\right)^{2} \right\}, & \text{ if } \left(\left|y\right| - 1\right)^{2} - 4\left(\tau - \left|y\right|\right) \geq 0, \\
			0, & \text{ otherwise}, \\
		\end{cases}
	\end{align*}	
	where
	\begin{equation*}
 		E = \left\{ 0 , \left[\frac{\left|y\right| - 1 + \sqrt{\left(\left|y\right| - 1\right)^{2} - 4\left(\tau - \left|y\right|\right)}}{2}\right]_{+} , \left[\frac{\left|y\right| - 1 - \sqrt{\left(\left|y\right| - 1\right)^{2} - 4\left(\tau - \left|y\right|\right)}}{2}\right]_{+} \right\},
	\end{equation*}
	with $\left[x\right]_{+} := \max\left\{ 0 , x \right\}$. 
\medskip
	
	Now we can apply CoCaIn BPG method and the function behavior is described in Figure \ref{fig:spurious_stationarity}. 

	\begin{figure}[t]
		\centering
		\begin{subfigure}[b]{0.45\textwidth}
			\centering
			\includegraphics[width=1\textwidth]{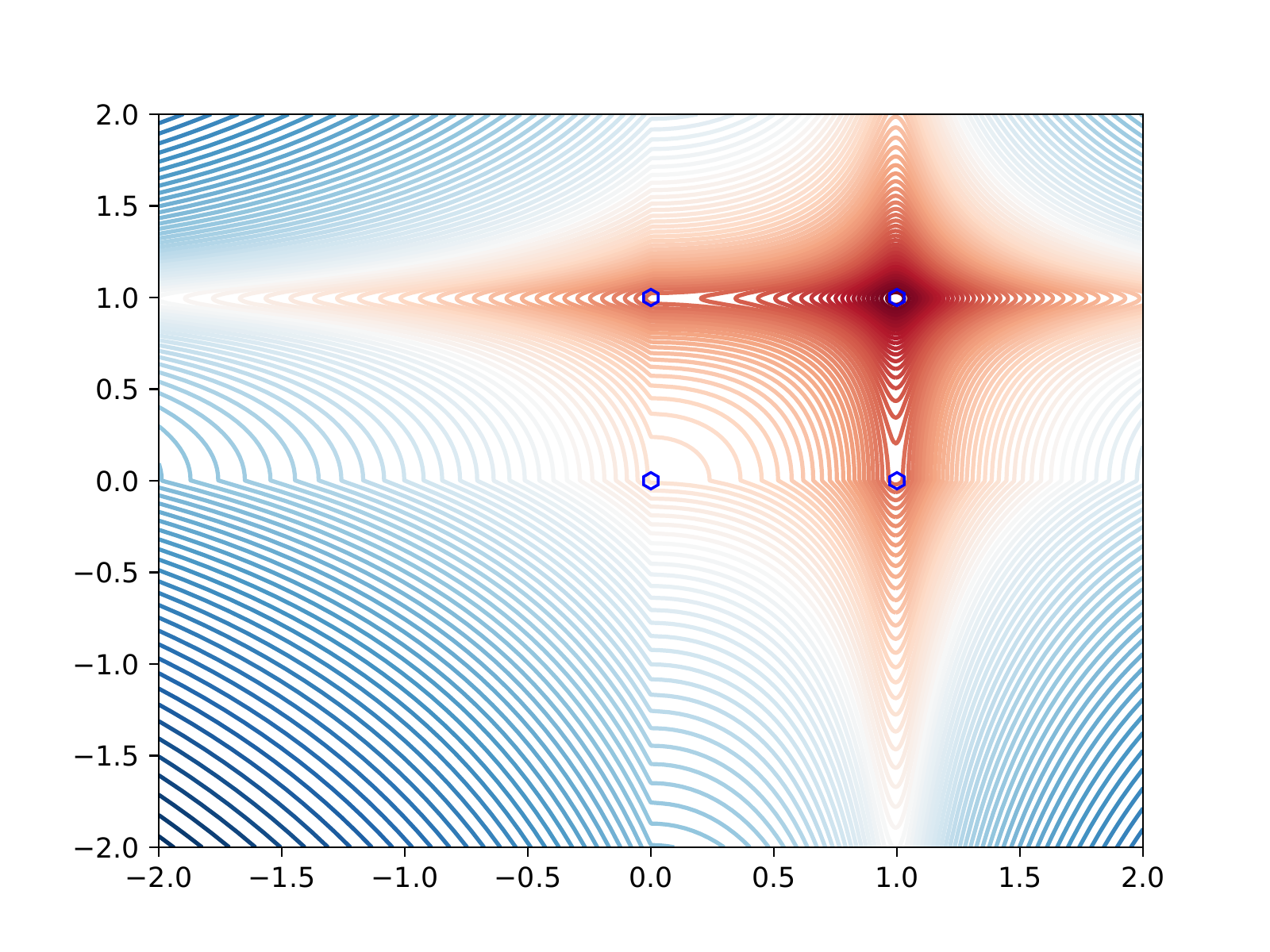}
			\caption{Function contour}
		\end{subfigure}
		\begin{subfigure}[b]{0.45\textwidth}
			\centering
			\includegraphics[width=1\textwidth]{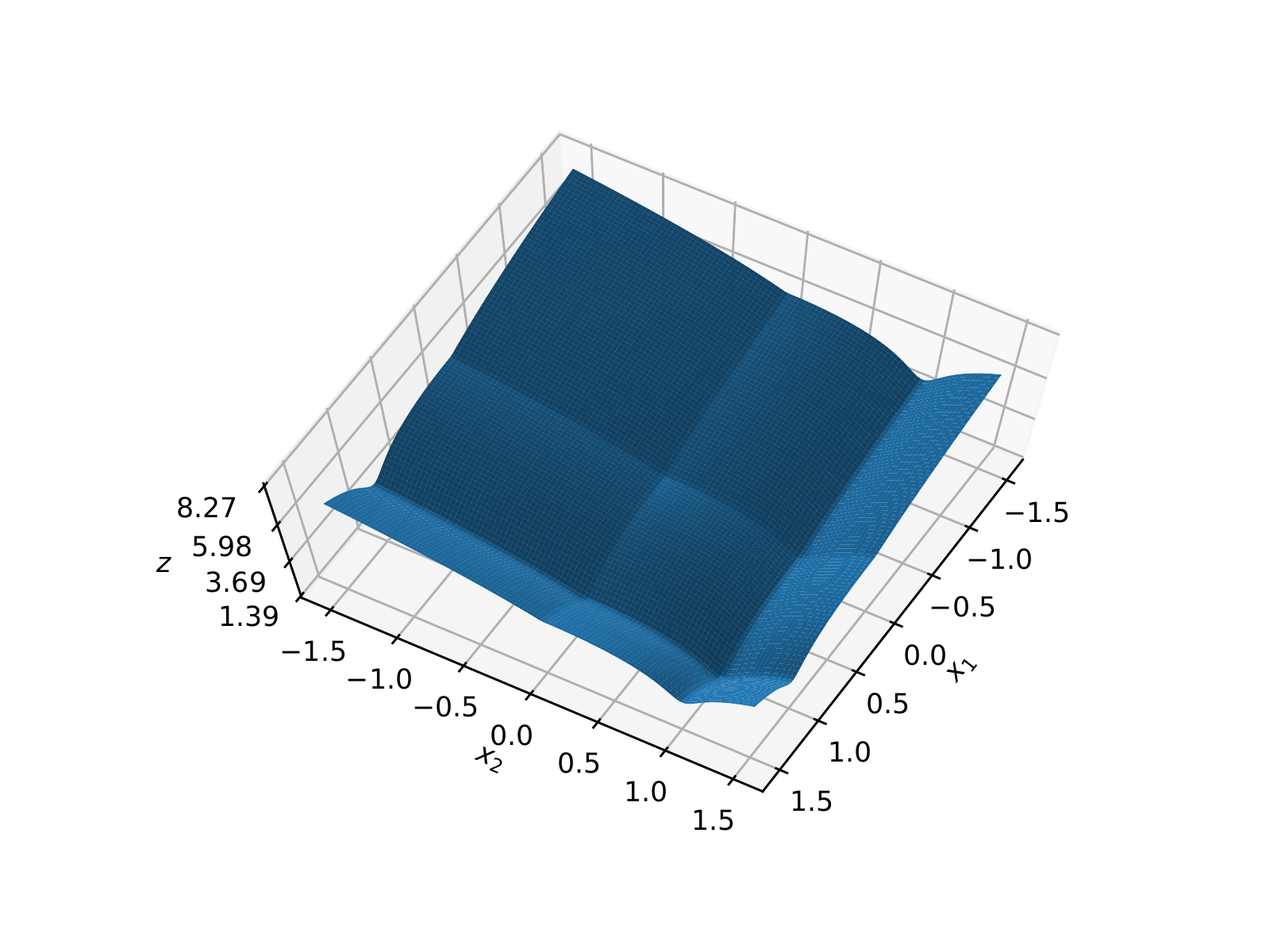}
			\caption{Function surface}
		\end{subfigure}
		\caption {\textbf{Function with spurious stationary points.} The left-hand side plot shows the contours of the objective function, and the four critical points (denoted with blue diamond). In the right-hand side plot, we show the objective function, where the $z$-axis represents the function value. Here, the critical points appear as downward kink.} 
		\label{fig:spurious_stationarity}
	\end{figure}

	The performance of CoCaIn BPG is illustrated in Figure \ref{fig:spurious_stationarity-2}, which shows that CoCaIn BPG can indeed escape spurious critical points to reach the global minimum. 
	
	\begin{figure}[htb]
		\centering
		\begin{subfigure}[b]{0.24\textwidth}
			\centering
			\includegraphics[width=1\textwidth]{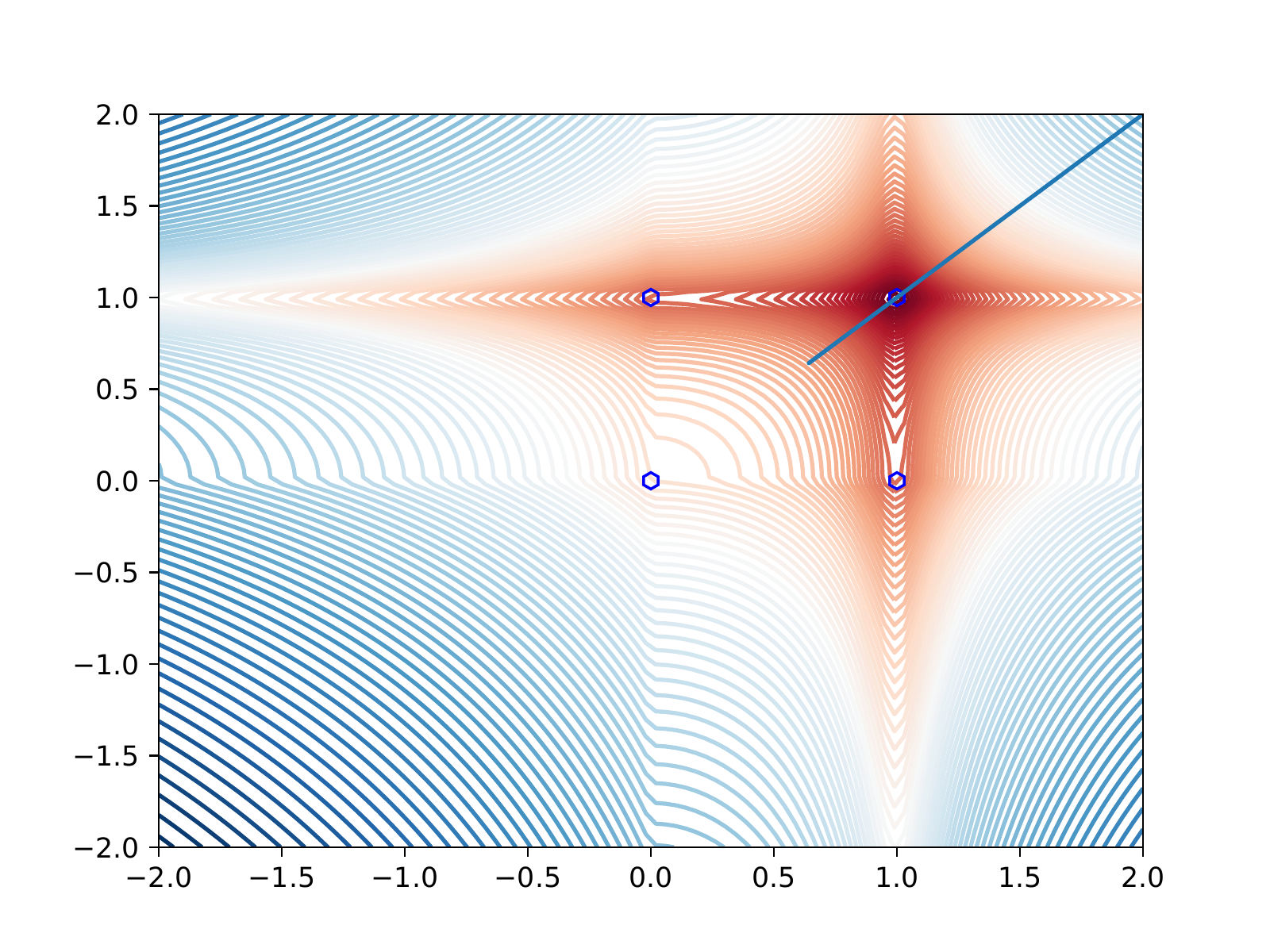}
			\caption{From $\left(2 , 2\right)$}
		\end{subfigure}
		\begin{subfigure}[b]{0.24\textwidth}
			\centering
			\includegraphics[width=1\textwidth]{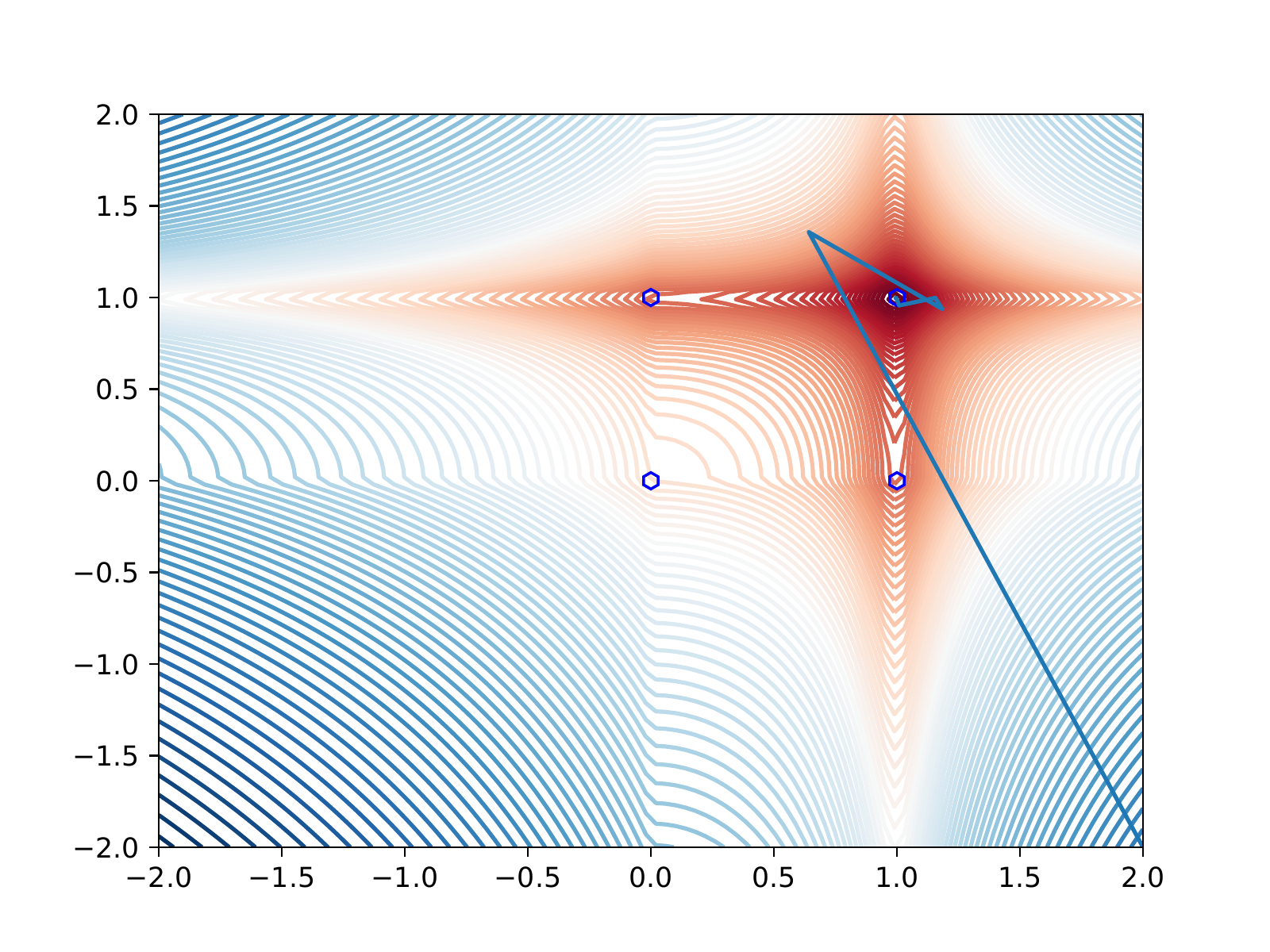}
			\caption{From $\left(-2 , 2\right)$}
		\end{subfigure}
		\begin{subfigure}[b]{0.24\textwidth}
			\centering
			\includegraphics[width=1\textwidth]{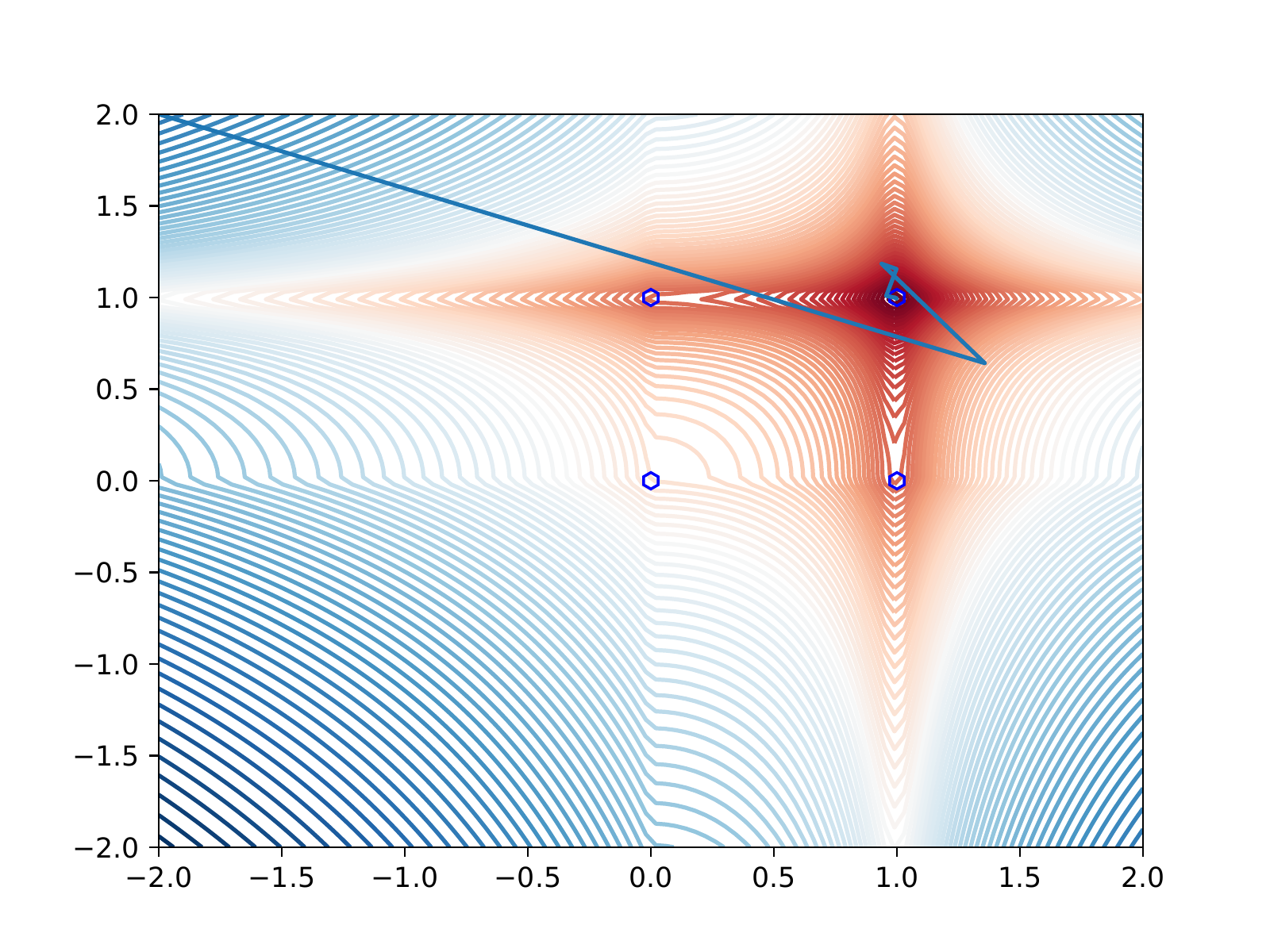}
			\caption{From $\left(2 , -2\right)$}
		\end{subfigure}
		\begin{subfigure}[b]{0.24\textwidth}
			\centering
			\includegraphics[width=1\textwidth]{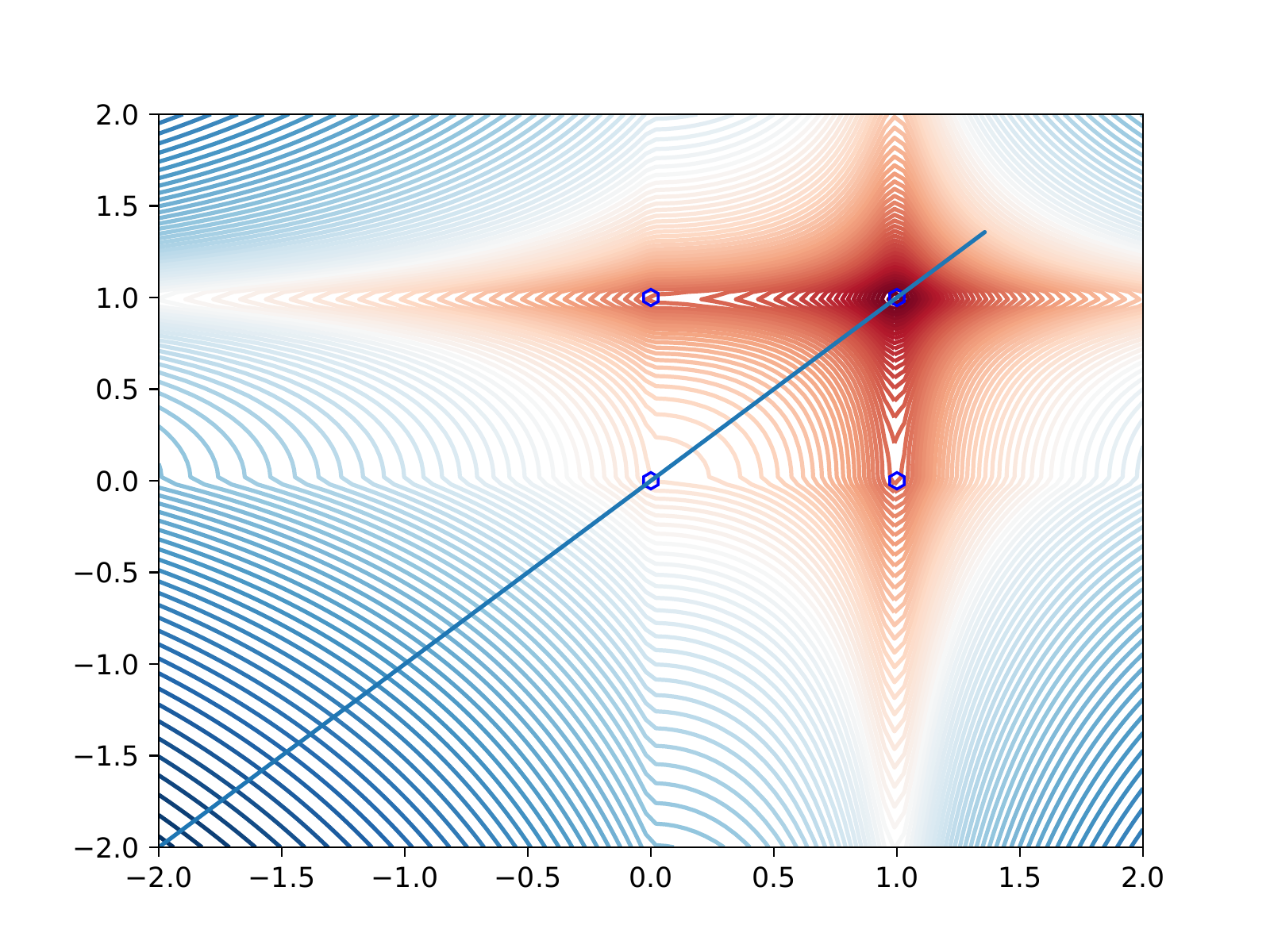}
			\caption{From $\left(-2 , -2\right)$}
		\end{subfigure}
		\caption {\textbf{CoCaIn can find the global minimum.} The CoCaIn BPG algorithm finds the global minimum at $\left(1 , 1\right)$, from various initialization points.} 
		\label{fig:spurious_stationarity-2}
	\end{figure}

\subsection{Quadratic Inverse Problems in Phase retrieval}	

	Phase retrieval has been an active research topic  for several years \cite{CLS2015, WGE2018, DR2017,L2017}. It gained a lot of attention from the optimization community, due to resulting hard non-convex problems \cite{BSTV2018, DR2017, CMLZ2018}. The phase retrieval problem can be described as follows. Given sampling vectors $a_{i} \in \real^{d}$, $i = 1 , 2 , \ldots , m$, and  measurements $b_{i} > 0$, we seek to find a vector $x \in \real^{d}$ such that the following system of quadratic equations is approximately satisfied,
	\begin{equation}
		\left| \act{a_{i} , x} \right|^{2} \approx b_{i}^{2}, \quad \forall \,\, i = 1 , 2 , \ldots , m.
	\end{equation}
	One typical way to tackle this system is by solving an optimization problem that seeks to minimize a certain error/noise measure in accomodating the equations. The objective function also depends on the type of noise \cite{CMLZ2018} in the system (for instance, Gaussian or Poisson noise). We assume additive Gaussian noise and the squared error measure  
	\begin{equation} \label{eq:wfo}
		\Psi\left(x\right) = f\left(x\right) + \frac{1}{4}\sum_{i = 1}^{m} \left(\act{a_{i} , x}^{2} - b_{i}^{2}\right)^{2},
	\end{equation}
	with 
	\begin{equation*}
		g\left(x\right) = \frac{1}{4}\sum_{i = 1}^{m} \left(\act{a_{i} , x}^{2} - b_{i}^{2}\right)^{2}.
	\end{equation*}
	The function $f$ acts as a regularizing term and is used to incorporate certain prior information on the wished solution. We conduct experiments with two options of regularizing functions: (i) squared $\ell_{2}$-norm, $f\left(x\right) = \left(\lambda/2\right)\norm{x}^{2}$ and (ii) $\ell_{1}$-norm,  $f\left(x\right) = \lambda\norm{x}_{1}$. When applying here the CoCaIn BPG method we use the following kernel generating distance function
	\begin{equation} \label{eq:first-kernel}
		h\left(x\right) = \frac{1}{4}\norm{x}_{2}^{4} + \frac{1}{2}\norm{x}_{2}^{2}\,.
	\end{equation}
 	We obviously have that $\dom{h} = \real^{d}$ and we record below a result \cite[Lemma 5.1, p. 2143]{BSTV2018}, which shows that the pair $\left(g , h\right)$ satisfies the L-smad property (see Definition \ref{D:Smad}).
	\begin{lemma} \label{L:ConsL-1}
		Let $g$ and $h$ be as defined above. Then, for any $L$ satisfying
		\begin{equation*}
			L \geq \sum_{i = 1}^{m} \left(3\norm{a_{i}a_{i}^{T}}^{2} + \norm{a_{i}a_{i}^{T}}\left|b_{i}^2\right|\right),
		\end{equation*}
		the function $Lh - g$ is convex on $\real^{d}$.
	\end{lemma}

	By the design of CoCaIn BPG algorithm, the inertial parameter $\gamma_k$ must satisfy \eqref{CoCaIn:3}. However, this  involves  backtracking over $\gamma_k$, which can computationally expensive for high dimensional problems. To this regard, following \cite{MWLCO2019}, we propose closed form expression for $\gamma_k$ which satisfies \eqref{CoCaIn:3}.  We also illustrate with our numerical experiments, that CoCaIn BPG variant with closed form inertia is  competitive to our main algorithm CoCaIn BPG. 

	\begin{lemma}[Closed form inertia]\label{lem:basic-qip-h}
		For $h$ defined in \eqref{eq:first-kernel}, we obtain the following gradient
		\begin{equation}\label{eq:qip-gradient}
			\nabla h(x) = (\norm{x}_2^2 + 1)x\,,
		\end{equation}
		and for any $a \in \R^d$, we have
		\begin{equation}\label{eq:upper-bound}
			\act{a,\nabla^2 h(x)a} \leq  \frac{3}{2}\norm{x}_2^2\norm{a}_2^2 + \frac{1}{2}\norm{a}_2^2\,.
		\end{equation} 
	\end{lemma}
	\begin{proof}
		Consider the expansion at $x + a$ till second order terms, we thus have
		\begin{align*}
			h(x+a) 	&= \frac{1}{4}\norm{x + a}_{2}^{4} + \frac{1}{2}\norm{x + a}_{2}^{2}\,,\\
					&= \frac{1}{4} \left(\norm{x}_2^2 + \norm{a}_2^2 + 2\act{a, x}\right)^2  + \frac{1}{2}\norm{x + a}_{2}^{2}\,,\\
					&= \frac{1}{4}\left( \norm{x}_2^4 + 4(\act{a, x})^2 + 4\norm{x}_2^2\act{a, x} + 2\norm{x}_2^2\norm{a}_2^2   \right) + \frac{1}{2}\left(\norm{x}_2^2 + \norm{a}_2^2 + 2\act{a, x} \right)\,.
		\end{align*} 
		The first order terms result in \eqref{eq:qip-gradient}		and we also have
		\begin{align*}
			\act{a,\nabla^2 h(x)a} 	&= \act{a, x}^2 + \frac{1}{2}\norm{x}_2^2\norm{a}_2^2 + \frac{1}{2}\norm{a}_2^2 \leq  \frac{3}{2}\norm{x}_2^2\norm{a}_2^2 + \frac{1}{2}\norm{a}_2^2\,,
		\end{align*}
		where the inequality follows due to Cauchy-Schwarz inequality.
	\end{proof}
	\begin{lemma}[\cite{MWLCO2019}]\label{lem:basic-ftc} Let $h \in \mathcal{G}(C)$ be twice continuously differentiable on $C$. Then, the following identity holds
		\[
			D_h(x^k,y^k) = \int_{0}^1 \left(1-t\right)\int_{0}^1 \act{ \nabla^2h\left(x^k + (t_1+ (1-t_1)t)(y^k-x^k)\right)(x^k-y^k), x^k-y^k}dt_1dt\,.
		\]
	\end{lemma}
	\begin{proposition}
		Denote $\Delta_k := {x^k - x^{k-1}}$, for any $k\geq 1$ the following holds
		\begin{equation*}
			D_h(x^k,y^k) \leq \gamma_k^2\norm{\Delta}^2 \left(\frac{3}{2}  \norm{x^k}^2 +  \frac{7}{4}\right)\,.
		\end{equation*}
	\end{proposition}
	\begin{proof}
		We use the strategy from \cite[Lemma 15]{MWLCO2019}. From Lemma~\ref{lem:basic-ftc}, we have
	\begin{align*}
		&\int_{0}^1 \left(1-t\right)\int_{0}^1 \act{ \nabla^2h\left(x^k + (t_1+ (1-t_1)t)(y^k-x^k)\right)(x^k-y^k), x^k-y^k}dt_1dt\\
		= & \gamma_k^2\int_{0}^1 \left(1-t\right)\int_{0}^1 \act{ \nabla^2h\left(x^k + (t_1+ (1-t_1)t)(y^k-x^k)\right)(x^k-x^{k-1}), x^k-x^{k-1}}dt_1dt \,,\\
		\leq& \gamma_k^2\int_{0}^1 \left(1-t\right)\int_{0}^1 \frac{3}{2} \norm{x^k-x^{k-1}}^2 \norm{x^k + (t_1+ (1-t_1)t)(y^k-x^k)}^{2}  dt_1dt \\
		& + \gamma_k^2\int_{0}^1 \left(1-t\right)\frac{1}{2}\norm{x^k-x^{k-1}}^2dt_1dt \,,\\
		\leq& \gamma_k^2\int_{0}^1 \left(1-t\right)\int_{0}^1 \left(3 \norm{x^k-x^{k-1}}^2 \norm{x^k}^2 + 3\norm{x^k-x^{k-1}}^{2}\right)  dt_1dt \\
		& + \gamma_k^2\int_{0}^1 \left(1-t\right)\frac{1}{2}\norm{x^k-x^{k-1}}^2dt_1dt \,,\\
		\leq& \gamma_k^2 \left(\frac{3}{2} \left( \norm{x^k-x^{k-1}}^2 \norm{x^k}^2 + \norm{x^k-x^{k-1}}^{2}\right)  + \frac{1}{4}\norm{x^k-x^{k-1}}^2\right)\,.
	\end{align*}
		where in the last step we used the upper bound \eqref{eq:upper-bound} from Lemma~\ref{lem:basic-qip-h}. Also, we used the following inequality
		\begin{align*}
			\norm{x^k + (t_1+ (1-t_1)t)(y^k-x^k)}^{2} 	&\leq 2\norm{x^k}^2 +  2(t_1+ (1-t_1)t)^{2}\gamma_k^2\norm{x^k - x^{k-1}}^2 \,,\\
														&\leq 2\norm{x^k}^2 +  2\norm{x^k - x^{k-1}}^2\,,
		\end{align*}
		where in the last step we used $\gamma_k^2 \leq 1$ and $(t_1+ (1-t_1)t)^{2} \leq 1$. With $\int_{0}^1(1-t)dt = \frac{1}{2}$ the result follows.\qedhere
	\end{proof}
	Therefore, in this case, Assumptions \ref{A:AssumptionA}, \ref{A:AssumptionB}, \ref{A:AssumptionC} and \ref{A:AssumptionD} are valid. We now discuss the update step of CoCaIn BPG, which requires the solution of the following subproblem
	\begin{equation} \label{eq:cocain-eq}
		x^{k + 1} \in \argmin_{x} \left\{ f\left(x\right) +\act{\nabla g\left(y^{k}\right) , x - y^{k}} + \frac{1}{\tau_{k}}D_{h}\left(x , y^{k}\right) \right\}.
	\end{equation}

	Following \cite{BSTV2018}, we provide closed form formulas for these optimization problems when $f$ is either the squared $\ell_{2}$-norm or the $\ell_{1}$-norm.

	\paragraph{$\ell_{1}$-norm.} Here we use the following closed form solution, derived in \cite[Proposition 5.1, p. 2145]{BSTV2018}. First, we define the soft-thresholding operator with respect to the parameter $\theta > 0$, as follows
	\begin{equation} \label{eq:soft}
		\SSS_{\theta}\left(y\right) = \argmin_{x \in \real^{d}} \left\{ \theta\norm{x}_{1} + \frac{1}{2}\norm{x - y}^{2} \right\} = \max\left\{ \left|y\right| - \theta , 0 \right\}\sgn\left(y\right)\,,
	\end{equation} 	
	 where all operations are applied coordinate-wise. Then the closed form solution of problem \eqref{eq:cocain-eq} is given by
	\begin{equation*}
		x^{k + 1} = t^{\ast}\SSS_{\lambda\tau_{k}}\left(\nabla h\left(y^{k}\right) - \tau_{k}\nabla g\left(y^{k}\right)\right),
	\end{equation*}
	where $t^{\ast}$ is the unique positive real root of the following cubic equation
	\begin{equation*}
		t^{3}\norm{\SSS_{\lambda\tau_{k}}\left(\nabla h\left(y^{k}\right) - \tau_{k}\nabla g\left(y^{k}\right)\right)}_{2}^{2} + t - 1 = 0\,.
	\end{equation*}

	\paragraph{Squared $\ell_{2}$-norm.} Using similar arguments as of \cite[Proposition 5.1, p. 2145]{BSTV2018}, we can easily derive that the solution of problem \eqref{eq:cocain-eq} is given by
	\begin{equation*}
		x^{k + 1} = t^{\ast}\left(\tau_{k}\nabla g\left(y^{k}\right) - \nabla h\left(y^{k}\right)\right),
	\end{equation*}
	where $t^{\ast}$ is the unique real root of the following cubic equation
	\begin{equation*}
		t^{3}\norm{\tau_{k}\nabla g\left(y^{k}\right) - \nabla h\left(y^{k}\right)}^{2} + \left(2\lambda\tau_{k} + 1\right)t + 1 = 0.
	\end{equation*}
	We illustrate, in Figure \ref{fig:quad-inverse:1}, the performance of CoCaIn BPG and CoCaIn BPG with closed form inertia (CoCaIn BPG CFI), compared with two other algorithms: (i) the Bregman Proximal Gradient Method with backtracking (denoted by BPG-WB) using the same kernel generating distance function (which is exactly CoCaIn BPG with $\gamma_{k} = 0$ for all $k \in \nn$) and (ii) the Inexact Bregman Proximal Minimization Line Search Algorithm (denoted by IBPM-LS) of \cite{OFB2019}. We also compare with the Bregman Proximal Gradient (BPG) method of \cite{BSTV2018} without backtracking and with the parameter $L$ as derived in Lemma \ref{L:ConsL-1}. 

	\begin{figure}[t]
		\begin{subfigure}[b]{0.245\textwidth}
			\includegraphics[width=\textwidth]{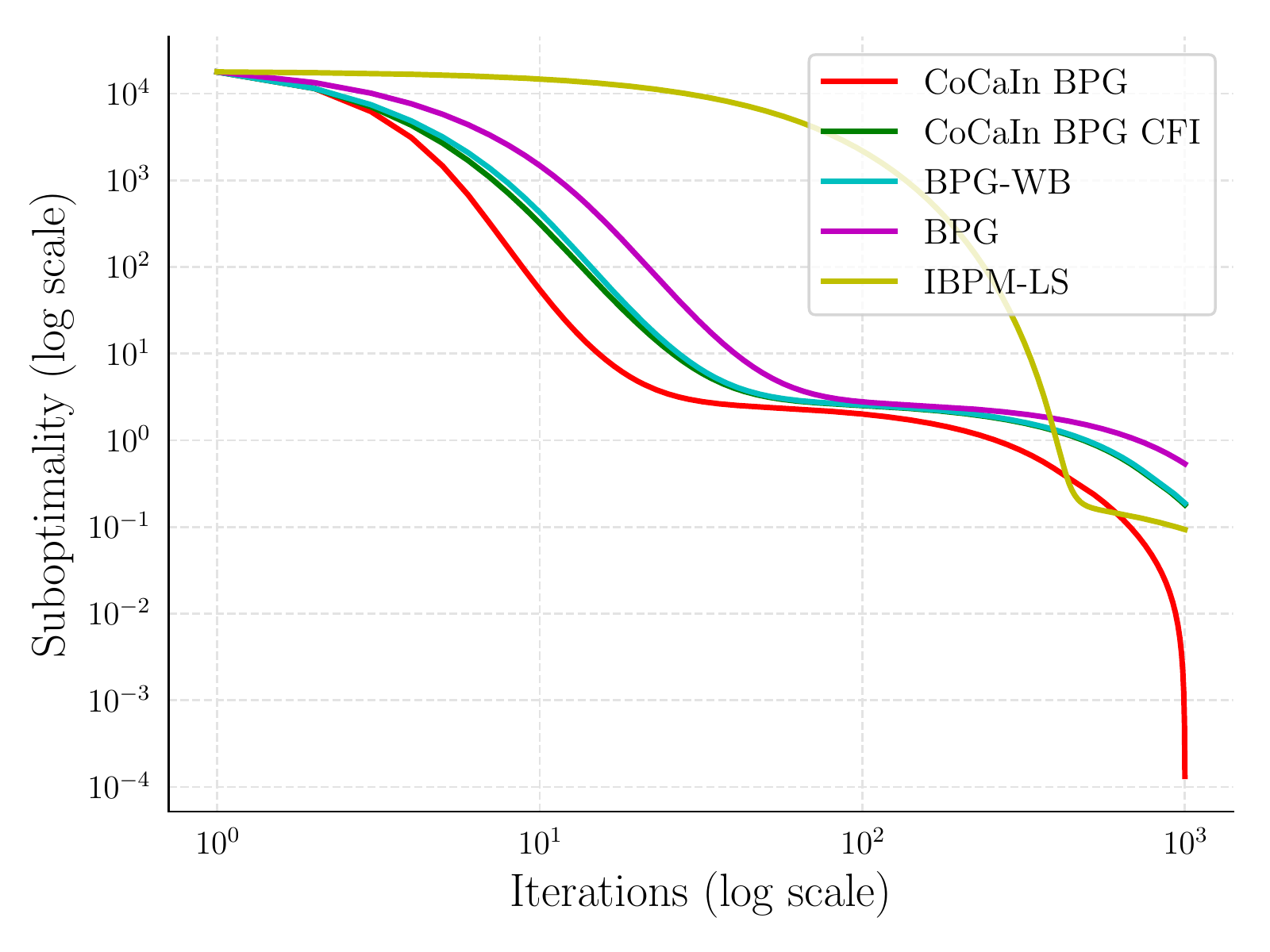}
			\caption{$\ell_{1}$-norm}
			\label{fig:quad-inverse-1}
		\end{subfigure}
		\begin{subfigure}[b]{0.245\textwidth}
			\includegraphics[width=\textwidth]{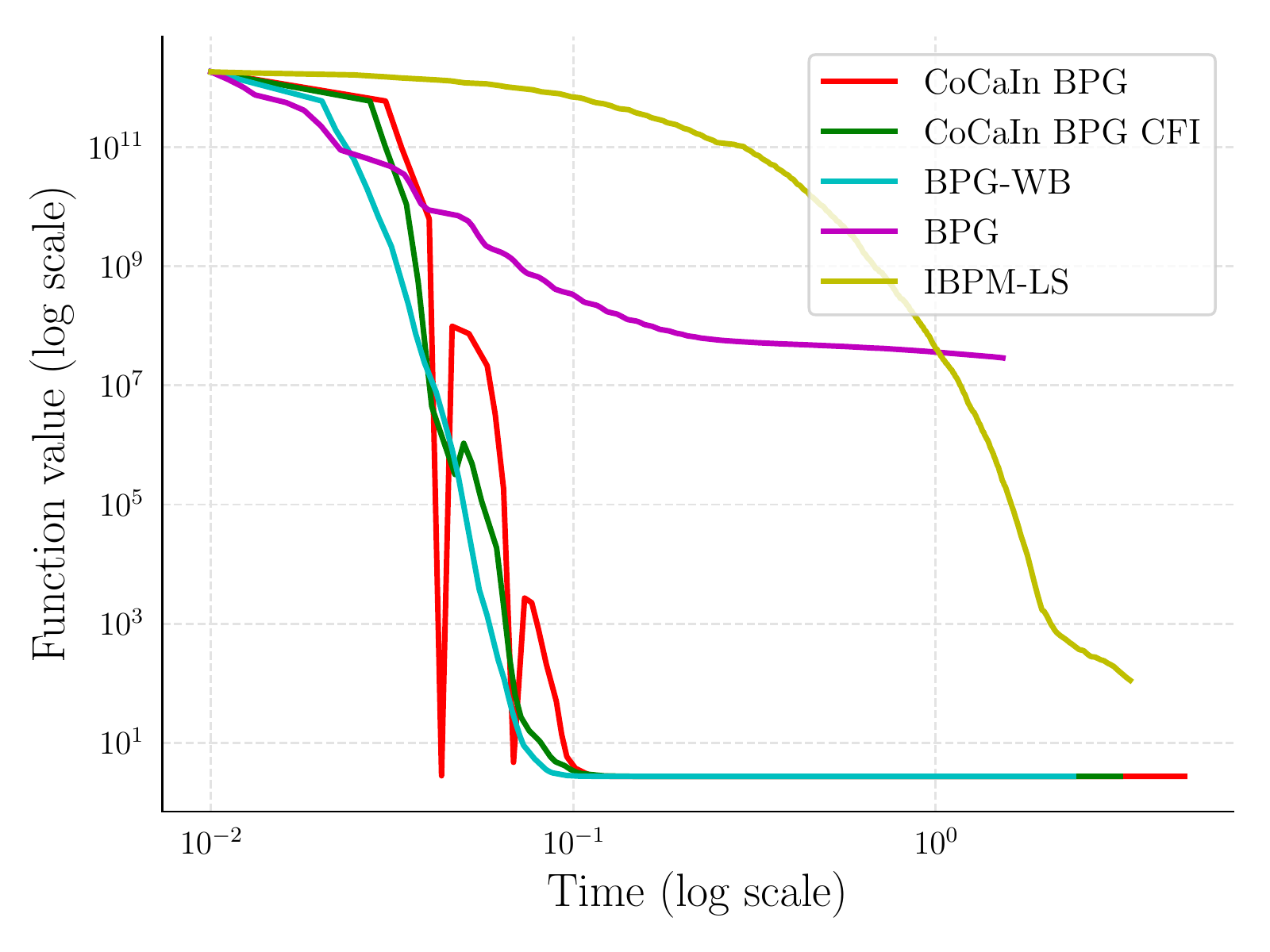}
			\caption{$\ell_{1}$-norm}
			\label{fig:quad-inverse-2}
		\end{subfigure}
		\begin{subfigure}[b]{0.245\textwidth}
			\includegraphics[width=\textwidth]{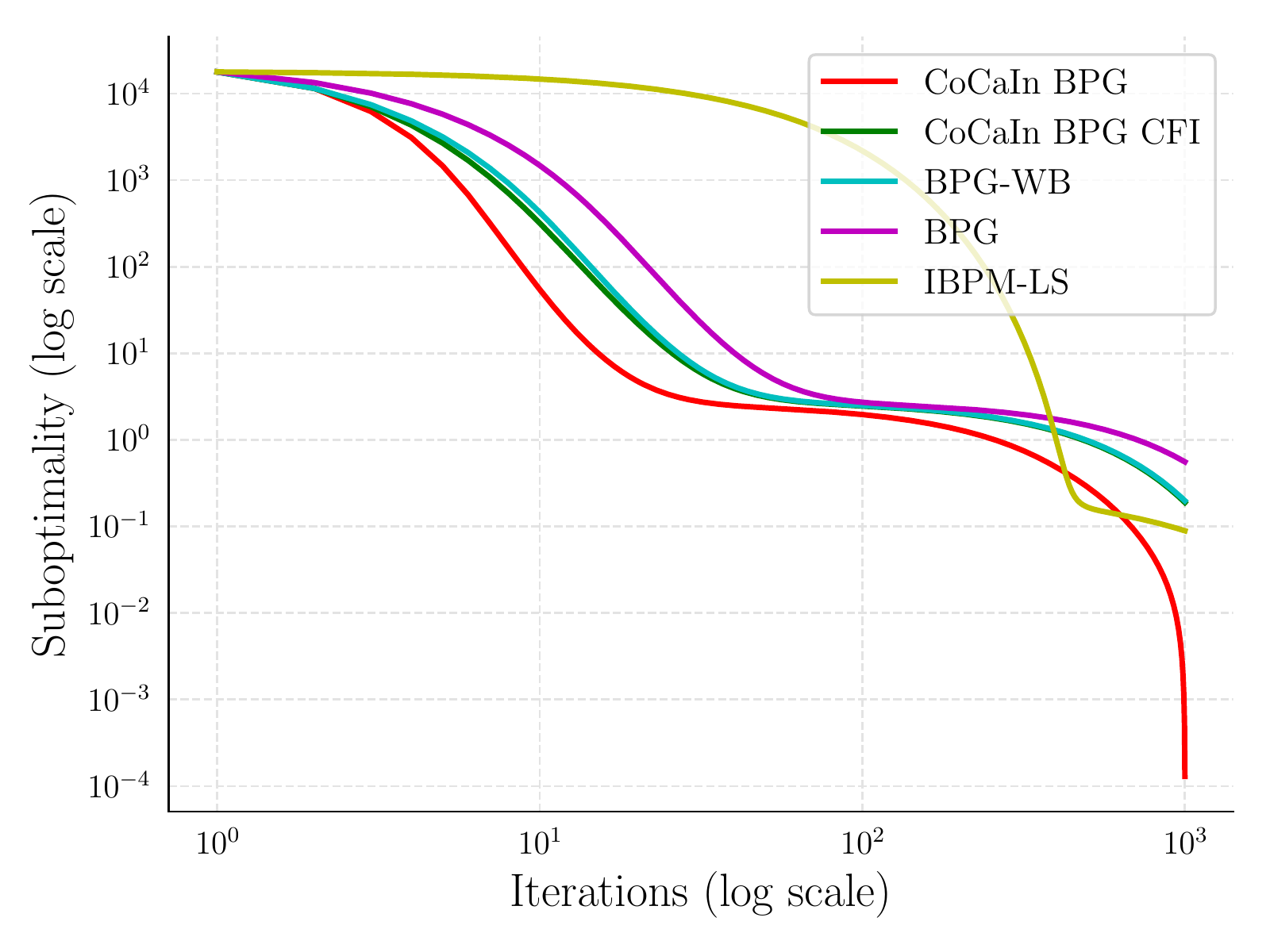}
			\caption{Squared $\ell_{2}$-norm}
			\label{fig:quad-inverse-3}
		\end{subfigure}
		\begin{subfigure}[b]{0.245\textwidth}
			\includegraphics[width=\textwidth]{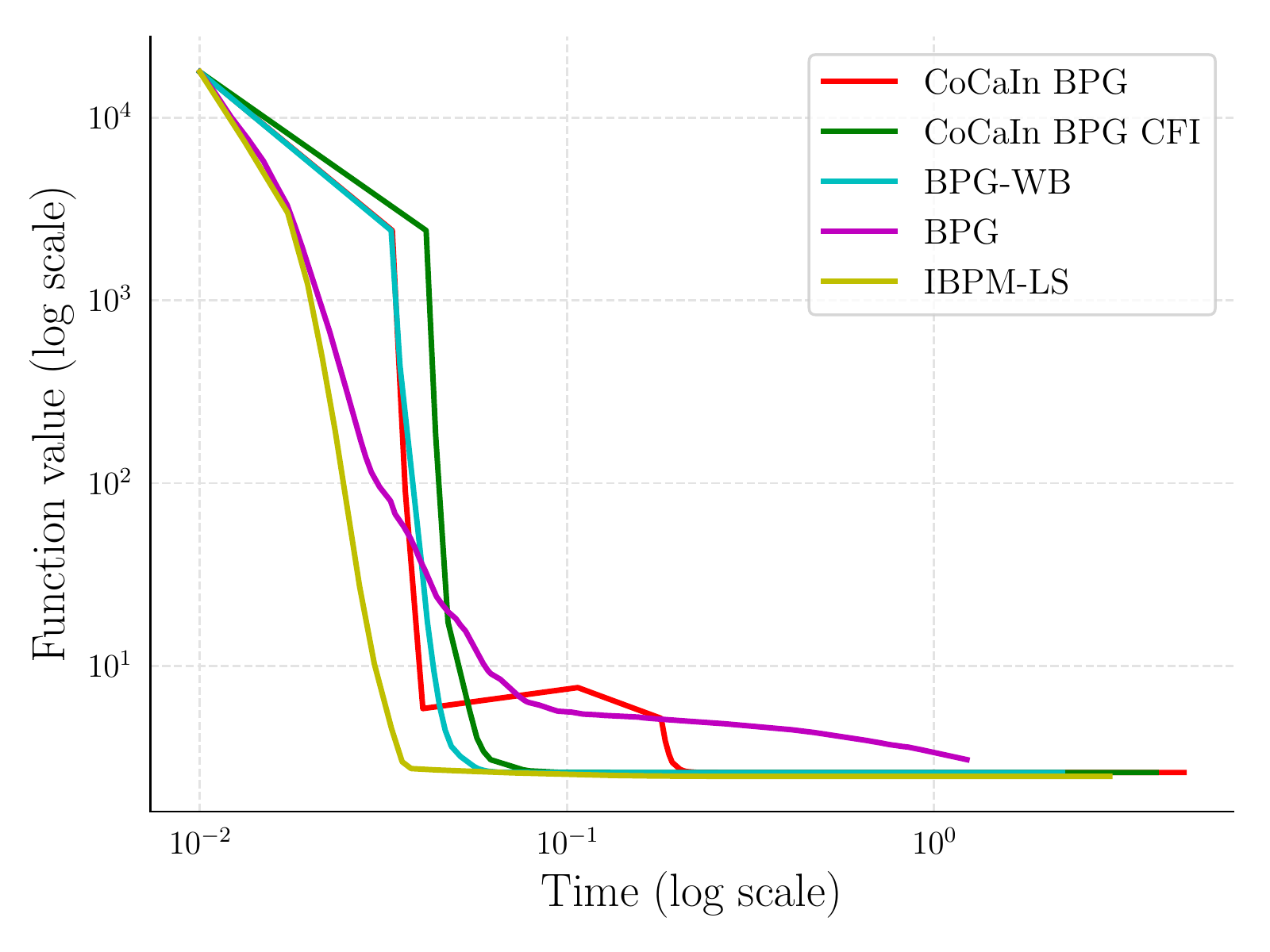}
			\caption{Squared $\ell_{2}$-norm}
			\label{fig:quad-inverse-4}
		\end{subfigure}
		\caption{\textbf{CoCaIn BPG for Phase Retrieval.} The plots illustrate that CoCaIn BPG, CoCaIn BPG CFI and BPG with Backtracking performances are competitive to other state of the art optimization algorithms. By suboptimality we mean the difference between the function value and the minimum function value attained by any of the algorithms. The difference is very significant when compared with BPG (without backtracking). This is due to the large $L$ used in the algorithm, thus resulting in smaller steps. On the other hand, CoCaIn BPG uses the local parameters $\lL_{k}$ and $\uL_{k}$, thus enjoys larger steps. The function values versus the time plots reveal that CoCaIn BPG rapidly attains a lower function value in a very early stage. 	 Note that CoCaIn BPG and CoCaIn BPG CFI perform very similarly, thus illustrating the benefits of closed form solutions. }
		\label{fig:quad-inverse:1}
	\end{figure}

\subsection{Non-convex Robust Denoising with Non-convex TV Regularization}
	We consider the problem of image denoising of a given image $b \in \R^{M \times N}$, where $M, N \in \mathbb{N} $. The goal is to obtain the true image, denoted by $x \in \real^{M \times N}$. However, in real world applications, it is possible that the measurements are noisy with outliers. The standard routine to deal with outliers is to use robust loss function.  The basic idea is to heavily penalize small errors and reasonably penalize large errors. This is  done to ensure that the predicted data $x$, is not influenced significantly by outliers. We consider a fully non-convex formulation of the problem, which includes a non-convex loss function along with a non-convex regularization. 
	\begin{figure}[t]
		\begin{subfigure}[b]{0.19\textwidth}
			\includegraphics[width=\textwidth]{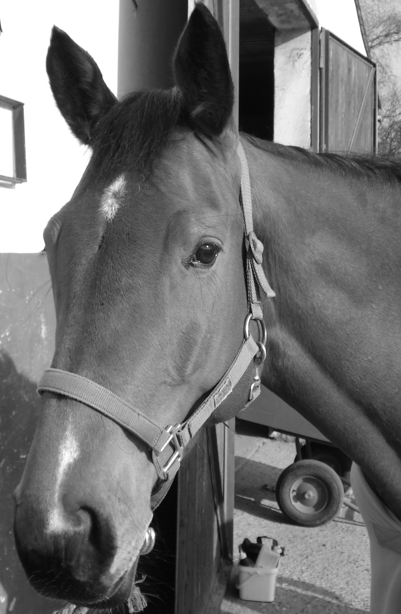}
			\caption{Ground truth}
			\label{fig:robust-denoising-1}
		\end{subfigure}
		\begin{subfigure}[b]{0.19\textwidth}
			\includegraphics[width=\textwidth]{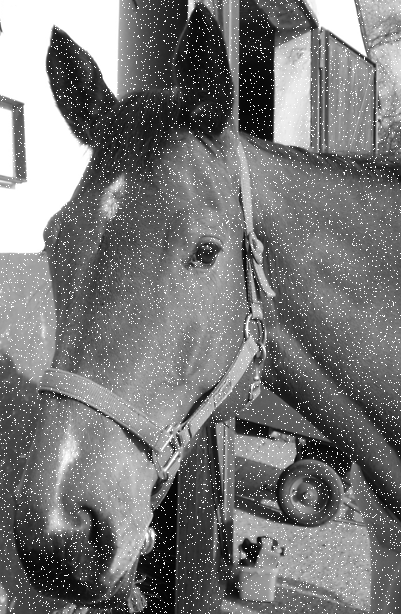}
			\caption{Noisy image}
			\label{fig:robust-denoising-2}
		\end{subfigure}
		\begin{subfigure}[b]{0.19\textwidth}
			\includegraphics[width=\textwidth]{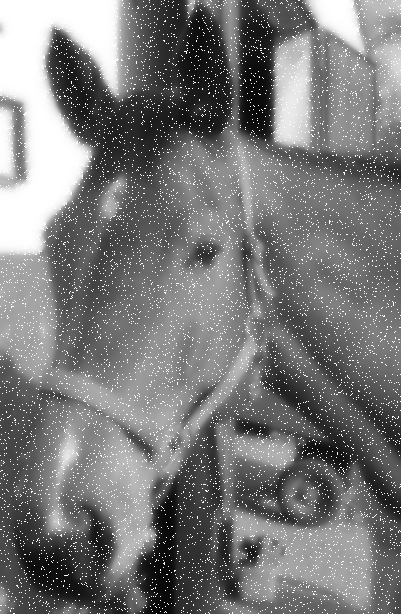}
			\caption{ $\ell_{2}$-data term}
			\label{fig:robust-denoising-4}
		\end{subfigure}
		\begin{subfigure}[b]{0.19\textwidth}
			\includegraphics[width=\textwidth]{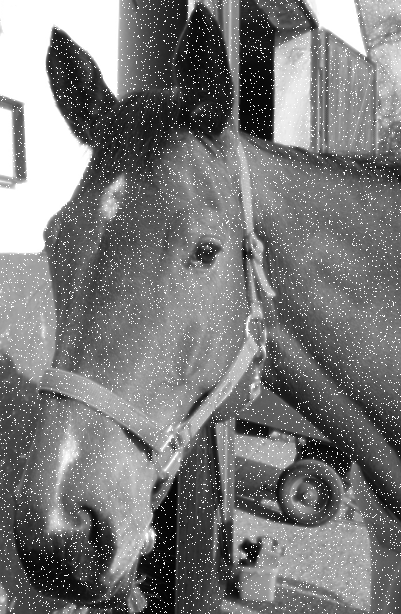}
			\caption{$\ell_{1}$-data term}
			\label{fig:robust-denoising-5}
		\end{subfigure}
		\begin{subfigure}[b]{0.19\textwidth}
			\includegraphics[width=\textwidth]{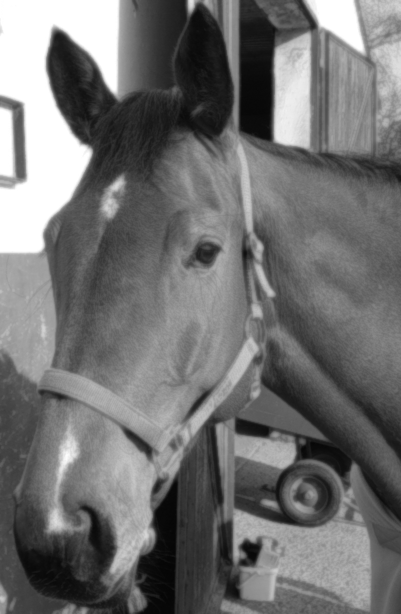}
			\caption{ Our setting}
			\label{fig:robust-denoising-6}
		\end{subfigure}
		\begin{subfigure}[b]{0.45\textwidth}
			\includegraphics[width=\textwidth]{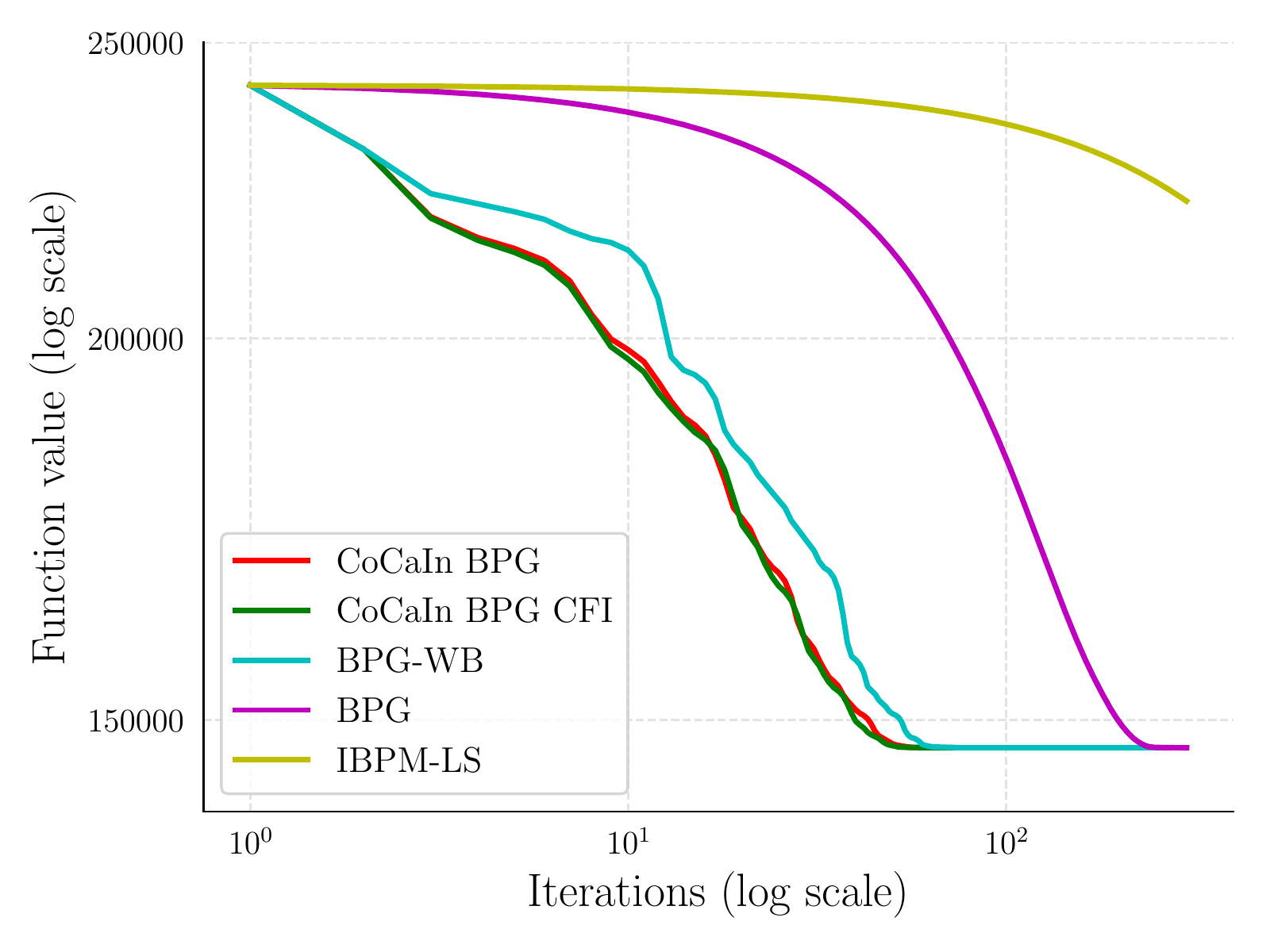}
			\caption{ Function value vs iterations}
			\label{fig:robust-denoising-7}
		\end{subfigure}
		\begin{subfigure}[b]{0.45\textwidth}
			\includegraphics[width=\textwidth]{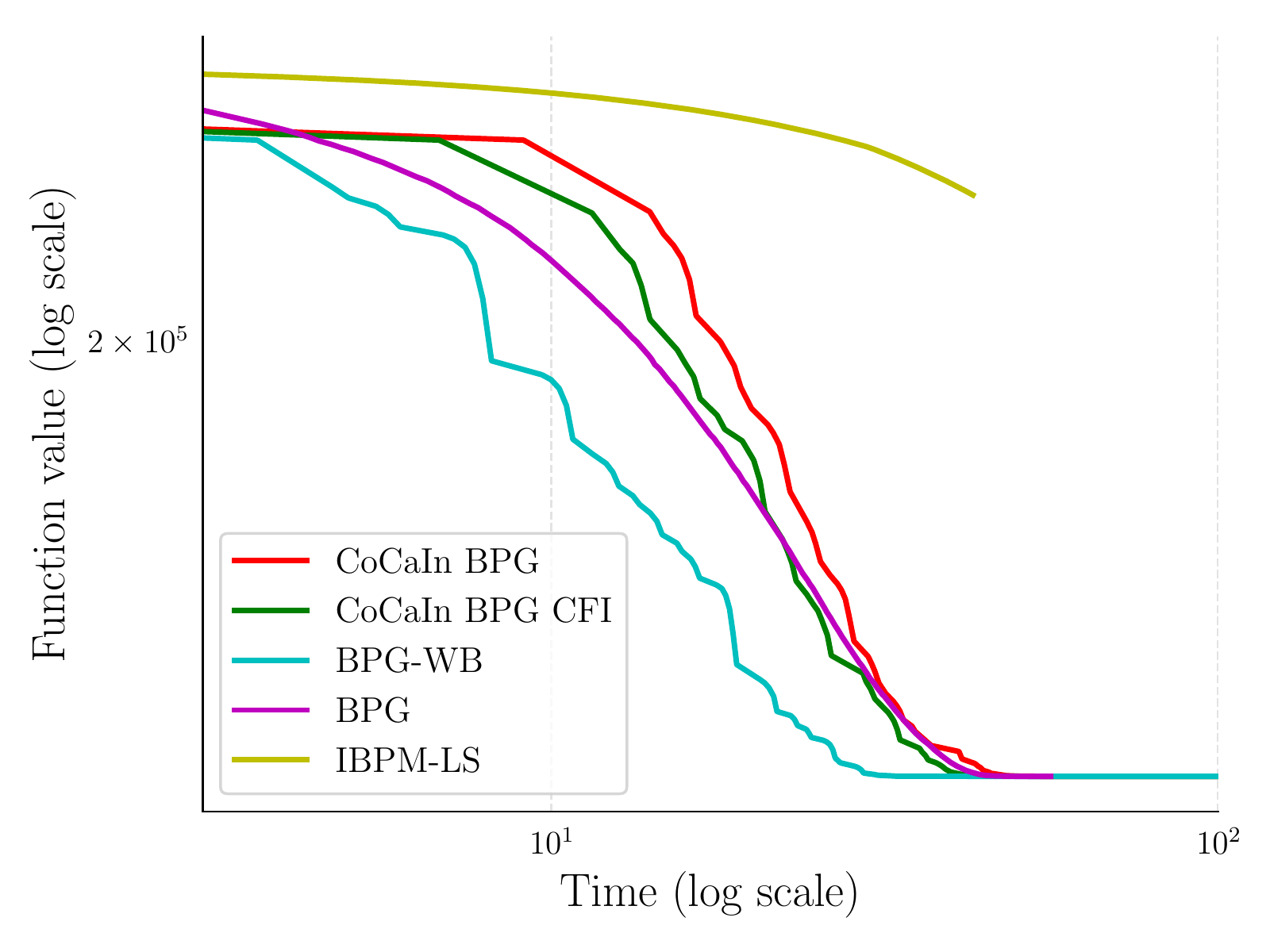}
			\caption{Function value vs Time }
			\label{fig:robust-denoising-8}
		\end{subfigure}
		
		\caption{\textbf{CoCaIn BPG for Robust Denoising.} We denote $\ell_2$-data term for the setting considered with  $f$ set to squared $\ell_2$-norm based loss and $g$ set to \eqref{eq:robust-2}. And, we denote  $\ell_1$-data term for the setting with  $f$ set to $\ell_1$-norm loss  and $g$ as in \eqref{eq:robust-2}. By our setting, we consider  \eqref{eq:robust-1} and \eqref{eq:robust-2}. The plots illustrate that BPG methods are competitive for the nonconvex robust image denoising problems.  IBPM-LS from \cite{OFB2019} is barely having any progress, due to flat surfaces. However, BPG methods do not have this issue. The plots illustrate that CoCaIn BPG performance is superior. Also, the reconstructed image obtained by applying CoCaIn BPG to our setting gives a robust reconstruction compared to other reconstructed images. }
		\label{fig:robust-denoising:0}
	\end{figure}	
	\medskip

	We need the following technical details to provide the full problem statement. The spatial finite difference operator is given by
	\begin{equation}
		(\mathcal{D}x)_{i,j} := \left((\mathcal{D}x)^1_{i,j}, (\mathcal{D}x)^2_{i,j}\right)
	\end{equation} 
	where $i \in [M]$ and $j \in [N]$.	The horizontal spatial finite differences are given by $(\mathcal{D}x)^1_{i,j}:= x_{i+1,j}-x_{i,j}$ for all $i<M$ and $0$ otherwise. The vertical spatial finite differences are given by $(\mathcal{D}x)^2_{i,j} := x_{i,j+1}-x_{i,j}$ for all $j < N$ and $0$ otherwise.
	\medskip
	
	The problem involves the following functions 
	\begin{align}
		f\left(x\right) &:= \sum_{i = 1}^{M}\sum_{j = 1}^{N} \log\left(1 + \left| x_{i,j} - b_{i,j}\right|\right)\,, \label{eq:robust-1}\\
		g\left(x\right) &:= \lambda\sum_{i = 1}^{M}\sum_{j = 1}^{N} \log\left(1 + \rho\norm{(\mathcal{D}x)_{i,j}}_2^2\right)\,,\label{eq:robust-2}
	\end{align}
	where $\lambda , \rho > 0$. The function $f$ is non-smooth non-convex and $g$ is smooth non-convex. The function $g$ is a non-convex variant of the popular Total Variation (TV) regularizer, which is used to prefer smooth signals while preserving sharp changes in the signal (such as edges of images). For an overview on non-convex regularizations we refer the reader to \cite{N2005,WCLQ2018}. Consider $h\left(x\right) = \left(1/2\right)\norm{x}^{2}_F$. It is easy to prove the convexity of $f\left(x\right) - \left(\alpha/2\right)\norm{x}^{2}_F$, by checking that its right derivative is monotonically increasing \cite[Theorem 6.4]{HL2012}, for all $\alpha \leq -1$. The function $Lh - g$ is convex for $L\geq 16\lambda\rho$.
	\medskip
	Due to separability of the function $f$, we can split the computation of the corresponding Bregman Proximal Gradient mapping, into the following separable subproblems
	\begin{align*}
		x_{i,j}^{k + 1} \in \argmin_{x_{i,j} \in \real} \left\{ \log\left(1 + \left|x_{i,j} - b_{i,j}\right|\right)  + \act{x_{i,j} - y_{i,j}^{k} , \nabla g(y^k)_{i,j}} + \frac{1}{2\tau_{k}}\left(x_{i,j} - y_{i,j}^{k}\right)^{2} \right\},
	\end{align*}
	which as discussed in Section \ref{SSec:Escape}, can be reduced to the computation of the proximal mapping of the function $\log\left(1 + \left|x - b\right|\right)$.
	\medskip

	We consider two additional experimental settings apart from our main setting given by \eqref{eq:robust-1} and \eqref{eq:robust-2}. Firstly, we use the $\ell_2$-norm based data term with the same regularization as in  \eqref{eq:robust-2}.  Secondly, we use the squared $\ell_1$-norm based data term with regularization as in \eqref{eq:robust-2}. We use the good image given in Figure~\ref{fig:robust-denoising-1} and add severe noise randomly of $10^5$ magnitude. We illustrate the robustness of the model given by \eqref{eq:robust-1} and \eqref{eq:robust-2} to such outliers. The reconstructed image from $\ell_2$-norm based data penalty term is given in Figure~\ref{fig:robust-denoising-4} and the reconstructed image from $\ell_1$-norm based data penalty term is given in Figure~\ref{fig:robust-denoising-5}, after applying CoCaIn BPG. Clearly the $\ell_1$-norm based data penalty is better than $\ell_2$-norm based data penalty term, which is due to the robustness properties of $\ell_1$-norm. However, even using $\ell_1$-norm is not enough in the presence of severe outliers, the robustness properties are not so significant. This is mitigated by our setting, where the reconstructed image is given in Figure~\ref{fig:robust-denoising-6}. In our setting, the data term in \eqref{eq:robust-1}  is very robust to outliers. In all the settings, we used $\lambda = 10$ and $\rho = 1$. The convergence plots for the experiments with \eqref{eq:robust-1} and \eqref{eq:robust-2} are given in Figure~\ref{fig:robust-denoising-7} and  \ref{fig:robust-denoising-8}. Note that CoCaIn BPG CFI uses the closed form inertia with Euclidean distance. BPG-WB  and BPG are same as in earlier experiments. IBPM-LS is a general purpose line-search algorithm for nonconvex nonsmooth problems proposed in \cite{OFB2019}. Even though, IBPM-LS is general, BPG based methods are much faster. The comparisons also illustrate that CoCaIn BPG is better in terms of convergence with respect 	to iterations and competitive with respect to time. CoCaIn BPG CFI performs very similar to CoCaIn BPG and as anticipated the time plots illustrate that  CoCaIn BPG CFI is slightly faster than CoCaIn BPG.

\section{Acknowledgments}
Mahesh Chandra Mukkamala and Peter Ochs acknowledge funding by the German Research Foundation (DFG Grant OC 150/1-1). Thomas Pock acknowledges support by the ERC starting grant HOMOVIS, No. 640156.

\section{Appendix: Proof of Theorem \ref{T:AbstrGlob}} \label{A:GlobConv}
	The set of all limit points of $\Seq{x}{k}$ is defined by
    \begin{equation*}
    		\omega\left(x^{0}\right) := \left\{ \overline{x} \in \real^{d}: \; \exists \mbox{ an increasing sequence of integers } \seq{k}{l} \mbox{ such that }\; x^{k_{l}} \rightarrow \overline{x} \mbox{ as } l \rightarrow \infty \right\}.
    \end{equation*}	
	We first prove the following result.
	\begin{lemma} \label{L:SubConv}
		Let $\Seq{x}{k}$ be a bounded gradient-like descent sequence for minimizing $\Psi_{\delta_1}$. Then, $\omega\left(x^{0}\right)$ is a nonempty and compact subset of $\crit \Psi$, and we have
		\begin{equation} \label{L:SubConv:1}
			\limit{k}{\infty} \dist\left(x^{k} , \omega\left(x^{0}\right)\right) = 0.
		\end{equation}
		In addition, the objective function $\Psi$ is finite and constant on $\omega\left(x^{0}\right)$.
	\end{lemma}
	\begin{proof}
		Since $\Seq{x}{k}$ is bounded there is $x^{\ast} \in \real^{d}$ and a subsequence $\left\{ x^{k_{q}} \right\}_{q \in \nn}$ such that $x^{k_{q}} \rightarrow x^{\ast}$ as $q \rightarrow \infty$ and hence $\omega\left(x^{0}\right)$ is nonempty. Moreover, the set $\omega\left(x^{0}\right)$ is compact since it can be viewed as an intersection of compact sets. Now, from conditions (C1) and (C3), and the lower semicontinuity of $\Psi$ (which follows from the lower semi-continuity of $f$ and $g$, see Assumption \ref{A:AssumptionA}), we obtain
		\begin{equation*}
       		\lim_{k \rightarrow \infty} D_{h}\left(x^{k - 1} , x^{k}\right) \leq \lim_{k \rightarrow \infty} \norm{x^{k} - x^{k - 1}}^{2} = 0
      	\end{equation*}
      	and therefore
		\begin{equation} \label{L:SubConv:2}
       		\lim_{q \rightarrow \infty} \Psi_{\delta_1}\left(x^{k_{q} + 1} , x^{k_{q}}\right) = \lim_{q \rightarrow \infty} \Psi\left(x^{k_{q}}\right) = \Psi\left(x^{\ast}\right).
      	\end{equation}
       	On the other hand, from conditions (C1) and (C2), we know that there is $w^{k+1} \in \partial \Psi_{\delta_1}\left(x^{k + 1} , x^{k}\right)$, $k \in \nn$, such that $w^{k + 1} \rightarrow \bo$ as $k \rightarrow \infty$. The closedness property of $\partial \Psi_{\delta_1}$ implies thus that $\bo \in \partial \Psi_{\delta_1}\left(x^{\ast} , x^{\ast}\right) = \left(\partial \Psi\left(x^{\ast}\right) , \bo\right)$. This proves that $x^{\ast}$ is a critical point of $\Psi$, and hence \eqref{L:SubConv:1} is valid.

		To complete the proof, let $\limit{k}{\infty} \Psi_{\delta_1}\left(x^{k + 1} , x^{k}\right) = l \in \real$. Then $\left\{ \Psi_{\delta_1}\left(x^{k_{q} + 1} , x^{k_{q}}\right) \right\}_{q \in \nn}$ converges to $l$ and from \eqref{L:SubConv:2} we have $\Psi\left(x^{\ast}\right) = l$. Hence the restriction of $\Psi_{\delta_1}$ to $\omega\left(x^{0}\right)$ equals $l$.
	\end{proof}
	We recall now the definition of the Kurdyka-{\L}ojasiewicz (KL) property \cite{K1998,L1963} and \cite{BDL2006} (for the non-smooth case). Denote $[\alpha < F < \beta] := \left\{ x \in \real^{d} : \; \alpha < F\left(x\right) < \beta \right\}$. Let $\eta > 0$, and set
	\begin{equation*}
		\Phi_{\eta} = \left\{ \varphi \in C^{0}[0 , \eta) \cap C^{1}(0 , \eta) : \; \varphi\left(0\right) = 0, \varphi \; \text{concave and} \; \varphi' > 0 \right\}.
	\end{equation*}
	\begin{definition}[The Non-smooth KL Property] \label{D:KL}
		A proper and lower semicontinuous function $F : \real^{d} \rightarrow \erl$ has the Kurdyka-{\L}ojasiewicz (KL) property locally at $\overline{u} \in \dom F$ if there exist $\eta > 0$, $\varphi \in \Phi_{\eta}$, and a neighborhood $U\left(\overline{u}\right)$ such that
		\begin{equation*}
  			\varphi'\left(F\left(u\right) - F\left(\overline{u}\right)\right)\dist\left(0 , \partial F\left(u\right)\right) \geq 1,
        \end{equation*}
		for all $u \in U\left(\overline{u}\right) \cap \left[F\left(\overline{u}\right) < F\left(u\right) < F\left(\overline{u}\right) + \eta \right]$.
  	\end{definition}
	Our last ingredient is a key uniformization of the KL property proven in \cite[Lemma 6, p. 478]{BST2014}, which we record below.
    \begin{lemma}[Uniformized KL Property] \label{L:KLProperty}
        Let $\Omega$ be a compact set and let $F : \real^{d} \rightarrow \erl$ be a proper and lower semicontinuous function. Assume that $F$ is constant on $\Omega$ and satisfies the KL property at each point of $\Omega$. Then, there exist $\tilde{\varepsilon} > 0$, $\eta > 0$ and $\varphi \in \Phi_{\eta}$ such that for all $\overline{x}$ in $\Omega$  one has,
        \begin{equation} \label{L:KLProperty:2}
            \varphi'\left(F\left(x\right) - F\left(\overline{x}\right)\right)\dist\left(0 , \partial F\left(x\right)\right) \geq 1,
        \end{equation}
        for all $x \in  \left\{ x \in \real^{d} : \; \dist\left(x , \Omega\right) < \tilde{\varepsilon} \right\} \cap \left[F\left(\overline{x}\right) < F\left(x\right) < F\left(\overline{x}\right) + \eta \right]$.
   	\end{lemma}	
	We can now restate and prove Theorem \ref{T:AbstrGlob}.
	\begin{theorem}
		Let $\Seq{x}{k}$ be a bounded gradient-like descent sequence for minimizing $\Psi_{\delta_1}$. If $\Psi$ and $h$ satisfy the KL property, then the sequence $\Seq{x}{k}$ has finite length, \ie $\sum_{k = 1}^{\infty} \norm{x^{k + 1} - x^{k}} < \infty$ and it converges to $x^{\ast} \in \crit \Psi$.
	\end{theorem}
    \begin{proof}
        Since $\Seq{x}{k}$ is bounded there exists a subsequence $\left\{ x^{k_{q}} \right\}_{q \in \nn}$ such that $x^{k_{q}} \rightarrow \overline{x}$ as $q \rightarrow \infty$. In a similar way as in Lemma \ref{L:SubConv} we get that
        \begin{equation} \label{T:AbstrGlob:1}
            \lim_{k \rightarrow \infty} \Psi_{\delta_1}\left(x^{k + 1} , x^{k}\right) = \lim_{k \rightarrow \infty} \Psi\left(x^{k}\right) = \Psi\left(\overline{x}\right).
        \end{equation}
        If there exists an integer $\bar{k}$ for which $\Psi_{\delta_1}\left(x^{\bar{k} + 1} , x^{\bar{k}}\right) = \Psi\left(\overline{x}\right)$ then condition (C1) would imply that $x^{\bar{k} + 1} = x^{\bar{k}}$. A trivial induction show then that the sequence $\Seq{x}{k}$ is stationary and the announced results are obvious. Since $\left\{ \Psi_{\delta_1}\left(x^{k + 1} , x^{k}\right) \right\}_{k \in \nn}$ is a nonincreasing sequence, it is clear from \eqref{T:AbstrGlob:1} that $\Psi\left(\overline{x}\right) < \Psi_{\delta_1}\left(x^{k + 1} , x^{k}\right)$ for all $k > 0$. Again from \eqref{T:AbstrGlob:1} for any $\eta > 0$ there exists a nonnegative integer $k_{0}$ such that $\Psi_{\delta_1}\left(x^{k + 1} , x^{k}\right) < \Psi\left(\overline{x}\right) + \eta$ for all $k > k_{0}$. From Lemma \ref{L:SubConv} we know that $\lim_{k \rightarrow \infty} \dist\left(x^{k} , \omega\left(x^{0}\right)\right) = 0$. This means that for any $\tilde{\varepsilon} > 0$ there exists a positive integer $k_{1}$ such that $\dist\left(x^{k} , \omega\left(x^{0}\right)\right) < \tilde{\varepsilon}$ for all $k > k_{1}$.
\medskip

		From Lemma \ref{L:SubConv} applied to $\Psi_{\delta_1}$, we know that $\omega\left(x^{0}\right)$ is nonempty and compact and that the function $\Psi$ is finite and constant on $\omega\left(x^{0}\right)$. Hence, we can apply the Uniformization Lemma \ref{L:KLProperty} applied to $\Psi_{\delta_1}$, which satisfies the KL property since $\Psi$ and $h$ do, with $\Omega = \omega\left(x^{0}\right)$. Therefore, for any $k \geq l := \max\left\{ k_{0} , k_{1} \right\} + 1$, we have
       	\begin{equation} \label{T:AbstrGlob:2}
        		\varphi'\left(\Psi_{\delta_1}\left(x^{k} , x^{k - 1}\right) - \Psi(\overline{x})\right)\dist\left(\bo , \partial \Psi_{\delta_1}\left(x^{k} , x^{k - 1}\right)\right) \geq 1.
        \end{equation}
       	This makes sense since we know that $\Psi_{\delta_1}\left(x^{k} , x^{k - 1}\right) > \Psi\left(\overline{x}\right)$ for any $k > l$. Combining \eqref{T:AbstrGlob:2} with condition (C2), see Proposition \ref{P:SubGB}, we get that
  		\begin{equation} \label{T:AbstrGlob:3}
        		\varphi'\left(\Psi_{\delta_1}\left(x^{k} , x^{k - 1}\right) - \Psi\left(\overline{x}\right)\right) \geq \rho_{2}^{-1}\left(\norm{x^{k - 1} - x^{k - 2}} + \norm{x^{k} - x^{k - 1}}\right)^{-1}.
       	\end{equation}
        For convenience, we define for all $p , q \in \nn$ and $\overline{x}$ the following quantity
        \begin{equation*}
            	\Delta_{p , q} : = \varphi\left(\Psi_{\delta_{1}}\left(x^{p} , x^{p - 1}\right) - \Psi\left(\overline{x}\right)\right) - \varphi\left(\Psi_{\delta_{1}}\left(x^{q} , x^{q - 1}\right) - \Psi\left(\overline{x}\right)\right).
        \end{equation*}
      	From the concavity of $\varphi$ we get that
       	\begin{equation} \label{T:AbstrGlob:4}
      		\Delta_{k , k + 1} \geq \varphi'\left(\Psi_{\delta_1}\left(x^{k} , x^{k - 1}\right) - \Psi\left(\overline{x}\right)\right)\left(\Psi_{\delta_1}\left(x^{k} , x^{k - 1}\right) - \Psi_{\delta_1}\left(x^{k + 1} , x^{k}\right)\right).
      	\end{equation}
      	Combining condition (C1) with \eqref{T:AbstrGlob:3} and \eqref{T:AbstrGlob:4} yields, for any $k > l$, that
       	\begin{equation*}
        		\Delta_{k , k + 1} \geq \frac{\norm{x^{k} - x^{k - 1}}^{2}}{\rho\left(\norm{x^{k - 1} - x^{k - 2}} + \norm{x^{k} - x^{k - 1}}\right)}, \quad \text{ where} \,\, \rho := \rho_{2}/\rho_{1}.
        \end{equation*}
        Using the fact that $2\sqrt{\alpha\beta} \leq \alpha + \beta$ for all $\alpha , \beta \geq 0$, we infer from the later inequality that
       	\begin{equation*}
        		4\norm{x^{k} - x^{k - 1}} \leq \norm{x^{k - 1} - x^{k - 2}} + \norm{x^{k} - x^{k - 1}} + 4\rho\Delta_{k , k + 1},
        \end{equation*}
        and thus
       	\begin{equation} \label{T:AbstrGlob:5}
        		3\norm{x^{k} - x^{k - 1}} \leq \norm{x^{k - 1} - x^{k - 2}} + 4\rho\Delta_{k , k + 1}.
        \end{equation}
       	Summing up \eqref{T:AbstrGlob:5} for $i = l + 2 , \ldots , k$ yields
		\begin{align*}
  			3\sum_{i = l + 2}^{k} \norm{x^{i} - x^{i - 1}} & \leq \sum_{i = l + 2}^{k} \norm{x^{i - 1} - x^{i - 2}} + 4\rho\sum_{i = l + 2}^{k} \Delta_{i , i + 1} \\
           	& \leq \sum_{i = l + 2}^{k} \norm{x^{i} - x^{i - 1}} + \norm{x^{l + 1} - x^{l}} + 4\rho\sum_{i = l + 2}^{k} \Delta_{i , i + 1} \\
           	& = \sum_{i = l + 2}^{k} \norm{x^{i} - x^{i - 1}} + \norm{x^{l + 1} - x^{l}} + 4\rho\Delta_{l + 2 , k + 1},
       	\end{align*}
       	where the last equality follows from the fact that $\Delta_{p , q} + \Delta_{q , r} = \Delta_{p , r}$ for all $p , q , r \in \nn$. Since $\varphi \geq 0$, recalling the definition of $\Delta_{l + 2 , k + 1}$, we thus have for any $k > l$ that
       	\begin{align*}
         	2\sum_{i = l + 2}^{k} \norm{x^{i} - x^{i - 1}} \leq \norm{x^{l + 1} - x^{l}} + 4\rho\varphi\left(\Psi_{\delta_1}\left(x^{l + 2} , x^{l + 1}\right) - \Psi\left(\overline{x}\right)\right),
        \end{align*}
        which implies that $\sum_{k = 1}^{\infty} \norm{x^{k + 1} - x^{k}} < \infty$, \ie $\Seq{x}{k}$ is a Cauchy sequence and hence together with Lemma \ref{L:SubConv}, we obtain the global convergence to a critical point.
    \end{proof}

\bibliographystyle{plain}
\bibliography{notes}
\end{document}